%% file: MAIN_AMS.tex
\documentclass[leqno]{amsart}
\title{Feynman Formula for Discrete-time Quantum Walks}

\author{Jean-Pierre Fouque}
\address{Department of Statistics and Applied Probability, South Hall, University of California, Santa Barbara, CA 93106, USA}
\email{fouque@pstat.ucsb.edu}
\thanks{Email for Fouque: \href{mailto:fouque@pstat.ucsb.edu}{fouque@pstat.ucsb.edu}}

\author{Tomoyuki Ichiba}
\address{Department of Statistics and Applied Probability, South Hall, University of California, Santa Barbara, CA 93106, USA}
\email{ichiba@pstat.ucsb.edu}
\thanks{Email for Ichiba: \href{mailto:ichiba@pstat.ucsb.edu}{ichiba@pstat.ucsb.edu}}

\author{Ka Lok Lam}
\address{Department of Statistics and Applied Probability, South Hall, University of California, Santa Barbara, CA 93106, USA}
\email{kalok@ucsb.edu}
\thanks{Email for Lam: \href{mailto:kalok@ucsb.edu}{kalok@ucsb.edu}}

\usepackage{paper}
\RequirePackage[authoryear, square]{natbib}
\renewcommand{\cite}{\citep}


\begin{document}

\begin{abstract}
We explicitly connect (discrete-time) quantum walks on $\Z$ with a four-state Markov additive process via a Feynman-type formula \eqref{eqns: Feynman Formula}. Using this representation, we derive a relation between the spectral decomposition of the Markov additive process
and the limiting density of the homogeneous quantum walk. In addition,  we consider a space-time rescaling of quantum walks, which leads to a system of quantum transport PDEs in continuous time and space with a phase interaction term. Our probabilistic representation for this type of PDE offers an efficient Monte Carlo computational technique. 
\end{abstract}

\maketitle

\input{Section1}

\input{Section2}

\input{Section3}

 \input{Section4}

 \input{Section5}

\input{Conclusion}

\appendix

\input{Appendix1}

\bibliographystyle{apalike}
\bibliography{bibtex}

\end{document}

%% file: Section1.tex
\section{Introduction}
\label{sec: introduction}
There are two types of quantum walks, continuous-time and discrete-time quantum walks. In this paper, we only discuss discrete-time quantum walks which were introduced by \cite{Aharonov93}  and began to draw attention to the quantum computing community since the paper \cite{ABMVW_2001}. We refer readers to the numerous surveys \cite{Kempe_2003}, \cite{AE_2012}, \cite{Konno_Survey_2020}, \cite{ZhouReview_2021}, \cite{KadianGarhwalKumar2021CSR}, ordered from the earliest to the most recent, and the references therein for detailed development of the literature of quantum walks.

Despite being defined to quantize classical random walks, quantum walks are deterministic dynamics and were also termed ``quantum cellular automata" by \cite{Meyer_1996}; they are well-known to behave very differently from classical random walks. The characteristic example is its famous ballistic long-time behavior with a bimodal weak limit (see for example \cite{konno_2005}, \cite{Grimmett_2004}),  which is in big contrast to the diffusive behavior driven by the classical central limit theorem. Thus, it is not surprising that techniques for studying quantum walks are rather non-probabilistic. They include, but are not limited to,  path-counting and orthogonal polynomials \cite{konno_2005}, Fourier transform \cite{Grimmett_2004}  \cite{MS4-2018} \cite{Strauch_2006}, the energy method \cite{MS-2020} and matrix analysis \cite{CGMV_2012} (see also the survey \cite{Konno_Survey_2020}).

In addition, the ballistic behavior described in the previous paragraph is also reflected by the possibility of a hyperbolic space-time rescaling of the quantum walk to a system of transport PDEs with a phase interaction term.  This was first stated and termed ``quantum transport PDEs'' in \cite{Molfetta_Debbasch_2012} (see also \cite{MS4-2018} for discussions of this type of Dirac PDEs). We refer the reader to the recent book \cite{Molfetta_2024} and references therein for developments in Physics between quantum walks and transport PDEs.

Our main contribution is to write a Feynman-type representation for quantum walks, bringing in probabilistic techniques for studying them. With our Feynman formula, we explicitly recognize the amplitude of a quantum walk to be a conditional expectation of a classical four-state Markov additive process. Using the formula, we are able to relate the spectral decomposition of the Markov additive process and the limiting density of quantum walks. In addition, we provide a probabilistic representation for the  quantum transport PDE in terms of continuous-time scaling limits of the aforementioned four-state Markov additive process. This offers an efficient Monte Carlo method for such PDE.

Connections between classical processes and quantum walks have been widely studied in the literature. One direction is the problem of mimicking. Since quantum walks are unitary dynamics, taking modulus squares to the amplitudes of states in each step gives a flow of probability marginals if the walk is initialized with a normalized state. It has been studied in \cite{Montero2017UniversalGenerators} and \cite{YSCRHN22} how to match this flow of marginals with a time-inhomogeneous Markov chain. In addition, it is well-established in the continuous-time literature (see \cite{Carlen1984Conservative}) that flow of marginals induced by a Schr\"{o}dinger equation can be mimicked by a homogeneous continuous-time Markov process, called a conservative diffusion. In addition, to incorporate the effect of decoherence to quantum walks, \cite{Attal_2012} introduced open quantum walks; by repeatedly measuring and running an open quantum walk, one naturally obtains quantum trajectories, which are classical stochastic processes. See the survey article \cite{Sinayskiy_2019} and the introduction section in \cite{CarboneGirottiMelchor2022JSP} for details. 
Finally, by introducing a walk-type zeta function, \cite{konno_zeta_2012}, \cite{konno_crossover_2020}, \cite{KKS_2022} recognized universal principles underlying both quantum walks and classical walks with Markov increments so that the walks can be handled in a unifying scheme. We remark that the 4-state Markov additive process induced from our Feynman formulae does not come from any measurement procedure and we intend to use it as a technique to uncover properties of quantum walks instead of comparing with quantum walks under a unifying scheme. Therefore, we believe that the relation between classical processes and quantum walks that we propose is novel and useful.

The paper is structured as follows.
In Section \ref{sec: Feynmanformula}, we recall the definition of a one-dimensional quantum walk, derive our main Feynman formula for the rotation coin $e^{i\theta \sigma_1}$, and comment on potential Monte Carlo simulations. Next, in Section \ref{sec: weaklimit}, we derive a relation between the spectral decomposition of the Markov additive process
and the limiting density of the homogeneous quantum walk , and we demonstrate the derivation of the weak limit for the long-term behavior of quantum walks. In Section \ref{sec: telegraph}, we consider a space-time rescaling of quantum walks, which leads to a system of transport PDEs in continuous time and space, with a phase interaction term.  Monte Carlo simulations are very efficient to approximate the solution of such a system. Finally, in Section \ref{sec: extensions},  we propose extensions to our Feynman formula to more general quantum walks, which include more general homogeneous coins, time-dependent coins (see for example \cite{Banuls2006}) and site-dependent coins (see for example \cite{WojcikEtAl2004}, \cite{LindenSharam2009}, \cite{MS4-2018} and \cite{MS-2020}).

%% file: Section2.tex
\section{Feynman formula for one-dimensional quantum walks}\label{sec: Feynmanformula}
\subsection{Notations}
For common notations in the quantum literature, see \cite{NielsenChuang_2010} and \cite{Hall_2013} for a more quantum computational and quantum mechanical account respectively. Here we mainly follow notations in \cite{Watrous_2018} with slight modifications.

Let $\H$ be a Hilbert space and denote  its  unit sphere  $S(\H):= \{x\in \H: \norm{x}_{\H}=1\}$. The space of bounded operators on $\H$ is denoted by $L(\H)$ and the space of unitary operators  by $U(\H)$. 

An alphabet is any set $\Sigma$ of countable cardinality. We denote 
\begin{align}
    \ell^2(\Sigma):= \{f:\Sigma\to \C\mid \norm{f}_{\ell_2}^2:= \sum_{x\in \Sigma}\abs{f(x)}^2<\infty \},
\end{align}
 the complex Hilbert space indexed by $\Sigma$ with inner product
\begin{align*}
    \inprod{f, g}:= \sum_{x\in \Sigma}f(x) \overline{g(x)},
\end{align*}
and standard orthonormal basis $\{e_i:x\mapsto \I_{x=i}\}_{i\in \Sigma}$; its respective dual basis is denoted by $\{e_{i}^*\}_{i\in \Sigma}$. If $\psi\in \ell_2(\Sigma)$, then we define $\{\psi(x)\}_{x\in \Sigma}$ via 
\begin{align*}
    \psi = \sum_{x\in \Sigma} \psi(x) e_x.
\end{align*}
Unless otherwise specified, any matrix representation of $L(\ell_2(\Sigma))$ is with respect to the standard basis $\{e_i\}_{i\in \Sigma}$. For the alphabet $\{\pm 1\}$, the ordered basis is always $\{e_1, e_{-1}\}$. Common unitary matrix representations for $L(\ell_2\{\pm 1\})$ include:
\begin{align}
    \sigma_1 :=
\begin{bmatrix}
    0 & 1\\
    1  & 0
\end{bmatrix}, \quad e^{i\theta\sigma_1}   = \begin{bmatrix}
    \cos \theta & i\sin \theta\\
    i\sin \theta & \cos \theta 
\end{bmatrix}\quad \mbox{for every $\theta\in \R$},
\end{align}
which is the Pauli-X matrix and its induced rotation matrix. Lastly, we recall the standard isomorphism result between $\ell^2(\Sigma_1)\otimes \ell^2(\Sigma_2)$ and $\ell^2(\Sigma_1\times \Sigma_2)$, and we denote the basis of both spaces by
$\{e_i e_j\}_{(i,j)\in \Sigma_1\times \Sigma_2}$, which is the same as $\{e_i \otimes e_j\}_{(i,j)\in \Sigma_1\times \Sigma_2}$ and $\{e_{i,j}\}_{(i,j)\in \Sigma_1\times \Sigma_2}$.

\subsection{One-dimensional quantum walks}
In this section we introduce the main object of the paper. 
\begin{definition}[Quantum walks on $\Z$]
 Quantum walks on $\Z$  are discrete-time dynamics of the form
\begin{align} \label{eq: def2.1}
    \psi_{n} = W\psi_{n-1}\quad \mbox{for}\quad n\geq 1; \quad \psi_0 \in L(\ell_2(\Z)\otimes \ell_2(\{\pm 1\})),
\end{align}
in which the walk operator $W:= SC\in L(\ell_2(\Z)\otimes \ell_2(\{\pm 1\}))$ is the composition of the coin operator $C = \sum_x e_xe_x^* \otimes C_x$ where $C_x\in U(\ell^2(\{\pm 1\}))$ are unitary for all $x\in \Z$ and the shift operator $S$, which is defined by $Se_x e_z = e_{x+z}e_z$ for every $x\in Z, z= \pm 1$. Note that the shift operator $S$ is also unitary and therefore $W$ itself is unitary.
\end{definition}
We remark that quantum walks are deterministic so we can visualize them as follows: suppose we have the initial state defined by $\psi_0 = \sum_x \alpha_x e_x e_1 + \sum_x \beta _x e_x e_{-1}$ where $\alpha_x, \beta_x\in \C$ for every $x\in \Z$. We represent the quantum walk as a line of vectors at each vertex of the integer lattice: 
\begin{align}
\begin{matrix}
\Z &\cdots & x-2 & x-1 & x & x+1 & x+2&\cdots \\[1em]
\begin{matrix}
e_
{1} \\[1em] e_{-1}
\end{matrix}
& \cdots &
\begin{pmatrix}
    \alpha_{x-2}\\[1em]
    \beta_{x-2}
\end{pmatrix}
& 
\begin{pmatrix}
    \alpha_{x-1}\\[1em]
    \beta_{x-1}
\end{pmatrix}
& 
\begin{pmatrix}
    \alpha_{x}\\[1em]
    \beta_{x}
\end{pmatrix}
& 
\begin{pmatrix}
    \alpha_{x+1}\\[1em]
    \beta_{x+1}
\end{pmatrix}
& 
\begin{pmatrix}
    \alpha_{x+2}\\[1em]
    \beta_{x+2}
\end{pmatrix}& \cdots \\[2em]
Coins &\cdots & C_{x-2} & C_{x-1} & C_x & C_{x+1} & C_{x+2}&\cdots 
\end{matrix}
\end{align}
where $C_x\in U(2)$ is the coin at position $x\in \Z$. 

If we denote the action of the coin by $
C_x\begin{pmatrix}
\alpha_x \\
\beta_x
\end{pmatrix}:=  \begin{pmatrix}
    \tilde \alpha_x \\
    \tilde \beta_x
\end{pmatrix}$, then, the shift step gives: 
\begin{align}
\begin{matrix}
\Z &\cdots & x-2 & x-1 & x & x+1 & x+2&\cdots\\[1em]
\begin{matrix}
e_1 \\[1em] e_{-1}
\end{matrix}
& \cdots &
\begin{pmatrix}
    \tilde \alpha_{x-3}\\[1em]
    \tilde \beta_{x-1}
\end{pmatrix}
& 
\begin{pmatrix}
    \tilde\alpha_{x-2}\\[1em]
    \tilde\beta_{x}
\end{pmatrix}
& 
\begin{pmatrix}
    \tilde\alpha_{x-1}\\[1em]
    \tilde\beta_{x_+1}
\end{pmatrix}
& 
\begin{pmatrix}
    \tilde\alpha_{x}\\[1em]
    \tilde\beta_{x+2}
\end{pmatrix}
& 
\begin{pmatrix}
    \tilde\alpha_{x+1}\\[1em]
    \tilde\beta_{x+3}
\end{pmatrix}& \cdots \\[2em]
Coins &\cdots & C_{x-2} & C_{x-1} & C_x & C_{x+1} & C_{x+2}&\cdots 
\end{matrix}
\end{align}
We remark that the shift operator acts on $(\tilde \alpha_x)_{x\in \Z}$ as a right shift, while acting on $(\tilde \beta_x)_{x\in \Z}$ as a left shift. Next we provide a concrete example.

 \begin{example} \label{ex: QWex2}
With the homogeneous rotation coin $C_{x} \equiv e^{i (\pi/4) \sigma_{1} }=
\frac{1}{\sqrt{2}}\begin{bmatrix}
    1 & i\\
    i & 1
\end{bmatrix}
$  for all $x \in \mathbb Z$  and the normalized initial condition $\psi_{0}(0,-1) = \psi_{0}(0, 1) = 1 / \sqrt{2}$, $\psi_0(x,\pm 1) \equiv 0$ for all $x \neq 0$, the first two steps of this quantum walk is represented in  Figure \ref{fig: quantum_walk_exmple}.

\end{example}

 \usetikzlibrary{arrows.meta, bending,positioning}
  \usetikzlibrary {fadings,patterns}
\tikzfading[name=fade out,
            inner color=transparent!60,
            outer color=transparent!100]
 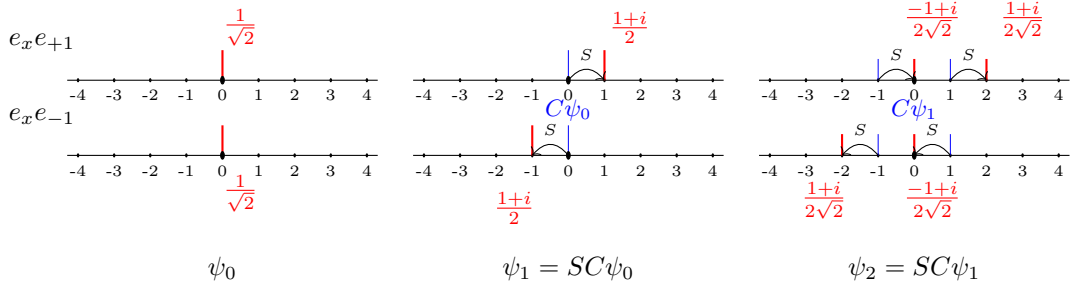
\begin{figure}
 \begin{center}
 \begin{tabular}{ccc}
\begin{tikzpicture}[xscale=0.48]
\draw[-] (-4.3,0) -- (4.3,0);
\draw[-] (-4.3,-1)--(4.3,-1);
\foreach \x in {-4,...,4}
\draw (\x,0) circle(.02);
\foreach \x in {-4,...,4}
\node[black] at (\x, -0.2) {\tiny\x};
\foreach \x in {-4,...,4}
\draw (\x,-1) circle(.02);
\foreach \x in {-4,...,4} 
\node[black] at (\x, -1.2) {\tiny\x};
\draw[thick, red] (0,0) -- (0,0.4);      
\node[red] at (0.5, 0.7) {$\frac{1}{\sqrt{2}}$};
\draw[thick, red] (0, -1) -- (0,-0.6);
\node[red] at (0.5, -1.5) {$\frac{1}{\sqrt{2}}$};
\node at (0,-2.5) {$\psi_{0}$};
\draw[fill] (0,0) circle(.05);
\draw[fill] (0,-1) circle(.05);
\node at (-5,.5) {$e_{x} e_{+1}$};
\node at (-5,-.5) {$e_{x} e_{-1}$};
\end{tikzpicture} & 
\begin{tikzpicture}[xscale=0.48]
\draw[-] (-4.3,0) -- (4.3,0);
\draw[-] (-4.3,-1)--(4.3,-1);
\foreach \x in {-4,...,4}
\draw (\x,0) circle(.02);
\foreach \x in {-4,...,4}
\node[black] at (\x, -0.2) {\tiny\x};
\foreach \x in {-4,...,4}
\draw (\x,-1) circle(.02);
\foreach \x in {-4,...,4}
\node[black] at (\x, -1.2) {\tiny\x};
\draw[thick, red] (-1,-1) -- (-1,-0.6); 
\draw[ultra thin, blue] (0,-1) -- (0,-0.6);
\node[red] at (-1.5, -1.7) {$\frac{1+i}{2}$};
\draw[thick, red] (1, 0) -- (1,0.4);
\draw[ultra thin, blue] (0, 0) -- (0,0.4);
\node[red] at (1.7, 0.7) {$\frac{1+i}{2}$};
\node at (0,-2.5) {$\psi_{1} = SC\psi_{0}$};
\draw[fill] (0,0) circle(.05); 
\draw[fill] (0,-1) circle(.05);
\draw[->] (0,-1) to [bend right] node[above]{\tiny $S$} (-1,-1); 
\draw[->] (0,0) to [bend left] node[above]{\tiny $S$} (1,0);
\node[blue] at (0, -0.4) {\small $C\psi_{0}$};
\end{tikzpicture} & 
\begin{tikzpicture}[xscale=0.48]
\draw[-] (-4.3,0) -- (4.3,0);
\draw[-] (-4.3,-1)--(4.3,-1);
\foreach \x in {-4,...,4}
\draw (\x,0) circle(.02);
\foreach \x in {-4,...,4}
\node[black] at (\x, -0.2) {\tiny\x};
\foreach \x in {-4,...,4}
\draw (\x,-1) circle(.02);
\foreach \x in {-4,...,4}
\node[black] at (\x, -1.2) {\tiny\x};
\draw[thick, red] (0,0-1) -- (0,0.283-1); 
\draw[ultra thin, blue] (1,0-1) -- (1,0.283-1);
\node[red] at (0.5, 0.9-2.5) {$\frac{-1+i}{2\sqrt{2}}$};
\draw[->] (1,0-1) to [bend right] node[above]{\tiny $S$} (0,0-1);
\draw[thick, red] (-2,0-1) -- (-2, 0.283-1);
\draw[ultra thin, blue] (-1,0-1) -- (-1, 0.283-1);
\node[red] at (-2.5, 0.9-2.5) {$\frac{1+i}{2\sqrt{2}}$};
\draw[->] (-1,0-1) to [bend right] node[above]{\tiny $S$} (-2,0-1);

\draw[thick, red] (0, 0) -- (0,1-0.717);
\draw[ultra thin, blue] (-1, 0) -- (-1,1-0.717);
\node[red] at (0.5,-1+1.8) {$\frac{-1+i}{2\sqrt{2}}$};
\draw[->] (-1,1-1) to [bend left] node[above]{\tiny $S$} (0,1-1); 
\draw[thick, red] (2, 1-1) -- (2,1-0.717);
\draw[ultra thin, blue] (1, 1-1) -- (1,1-0.717);
\node[red] at (3, -1+1.8) {$\frac{1+i}{2\sqrt{2}}$};
\draw[->] (1,1-1) to [bend left] node[above]{\tiny $S$} (2,1-1);
\node at (0,-2.5) {$\psi_{2} = SC\psi_{1}$};
\draw[fill] (0,0) circle(.05); 
\draw[fill] (0,-1) circle(.05);
\node[blue] at (0, -0.4) {\small $C\psi_{1}$};
\end{tikzpicture} 
\end{tabular}
\caption{The mechanism of the one-dimensional quantum walk $ (\psi_{n} = W^{n} \psi_{0}, n = 0,1,2 )$ with the rotation coin $e^{i (\pi/4) \sigma_{1}} 
$ and the initial condition $\psi_{0}(0,-1) = \psi_{0}(0, 1) = 1 / \sqrt{2}$ in Example \ref{ex: QWex2}. Two coin states $z = -1$ and $z=+1$ are assigned for each position $x \in  \mathbb Z$. For each step, we apply the coin operator $C$ first and then the shift operator $S$. At $n = 2$, the state $C\psi_1$ has 4 non-zero components because the action of $C$ splits each component of $\psi_1$ into two.}  \label{fig: quantum_walk_exmple}
\end{center}
\end{figure}

\subsection{Deriving the Feynman formula}\label{sec: Feynman_Formula}
In this section we derive the Feynman formula we advertised. We begin with the Feynman-path representation: 
\begin{lemma}\label{lemma: Poisson-like_coefficients}
For a homogeneous quantum walk on the 1-D lattice with coin $C_x = e^{i\theta \sigma_1}$, we have 
\begin{align}\label{eqns: Poisson-like_coefficients}
    \psi_n(x, z)&= \sum_{k_1, \cdots, k_n =0}^\infty \frac{(i\theta)^{k_1+\cdots+k_n}}{k_1!\cdots k_n!} \psi_0(x_n, z_n),
\end{align}
where  $(x_n, z_n)$ are being defined by $z_m:= (-1)^{k_m}z_{m-1}$ and $x_m:= x_0-\sum_{j=0}^{m-1} z_j $ for $m\geq 1$ with $(x_0, z_0):=(x, z)$.
\end{lemma}
\begin{proof}
    First, applying the walk operator on a basis element gives
\begin{align*}
   We_x e_z =  S e^{i\theta \sigma_1} e_xe_z = S \sum_{k=0}^\infty \frac{(i\theta)^k}{k!}e_x \sigma_1^k e_{ z}= \sum_{k=0}^\infty \frac{(i\theta)^k}{k!}e_{x+(-1)^k z} e_{(-1)^k z}. 
\end{align*}
Then for any state $\psi$, we obtain from 
\begin{align*}
    W\psi &= W \sum_{x\in \Z, z\in \{\pm 1\}} \psi(x, z) e_x e_z = \sum_{x\in \Z, z\in \{\pm 1\}, k\in \N} \frac{(i\theta)^k}{k!}\psi(x, z) e_{x+(-1)^k z} e_{(-1)^k z}\\
    &=\sum_{x\in \Z, z\in \{\pm 1\}, k\in \N} \frac{(i\theta)^k}{k!}\psi(x-z, (-1)^kz) e_{x} e_{z} \notag 
\end{align*}
that 
\begin{align}\label{eqns: one-step-proof}
    (W\psi)(x, z) = \sum_{k=0}^\infty \frac{(i\theta)^k}{k!}\psi(x-z, (-1)^k z).
\end{align}
 Repeatedly applying \eqref{eqns: one-step-proof} on $\psi_n:= W^n \psi_0$ where $n\geq 1$ gives 
\begin{align}
    \psi_n(x, z) = \sum_{k_1, \cdots , k_n}^\infty \frac{(i\theta)^{k_1+\cdots + k_n}}{k_1! \cdots k_n!} \psi(x-z-(-1)^{k_1}z - \cdots -(-1)^{k_1+\cdots +k_{n-1}}z, (-1)^{k_n}z)
\end{align}
which is exactly the required identity. 
\end{proof}

\begin{remark}
The coin $C = e^{i\theta \sigma_1}$ is periodic in $\theta$ with period $2\pi$. In the following we restrict $\theta\in (0,2\pi)$ where we exclude the trivial case $\theta=0\,\pmod{2\pi}$.
\end{remark}
\begin{remark}
    Note that  \eqref{eqns: Poisson-like_coefficients}
can be re-written as 
\begin{align}\label{eqns: Poisson-like-new}
    \psi_n(x_0, z_0)&= \sum_{k_1, \cdots, k_n =0}^\infty i^{k_1+\cdots+k_ n}\frac{\theta^{k_1}\cdots\theta^{k_n}}{k_1!\cdots k_n!} \psi_0(x_n, z_n).
\end{align}
The Poisson-like coefficients in \eqref{eqns: Poisson-like-new} and the fact that the factor $i^{k_1+\cdots + k_n}$ depends only on $(k_1+\cdots + k_n)\pmod 4$ motivate the following definition.
\end{remark}

\begin{definition}\label{defn: MA Model} 
Introducing a sequence of i.i.d. Poisson random variables $(N_1,N_2,\cdots)$ with parameter $\theta\in (0,2\pi)$, we define the Markov additive model\footnote{A Markov additive model is a triple that consists of a Markov chain, a real-valued function of it and a sum of that real-valued function of the Markov chain (see for instance \cite{NeyNummelin_1987a} and \cite{NeyNummelin_1987b}).} $(S_n, Z_n:= f(S_n), X_n)\in \{0, 1, 2, 3\} \times \{\pm 1\}\times\Z$ by
\begin{align}
    &S_0= 0, \quad S_n = \sum_{j=1}^n N_j \pmod{4} \quad \mbox{for}\quad n\geq 1,\\
    &Z_0=z,\quad Z_n=f(S_n)Z_0, \quad f(s):= (-1)^s  \quad\mbox{for} \quad s\in \{0, 1, 2, 3\} \quad\mbox{and}\quad n\geq 1,\\
    &X_0=x, \quad X_n=X_{n-1}-Z_{n-1}=\cdots=
    X_0 - \sum_{j=0}^{n-1}Z_j\quad\mbox{for}\quad n\geq 1.
\end{align}

\end{definition}

Our main result is as follows: 
\begin{theorem}[Feynman Formula]\label{thm: Feynman_Formula}
For the homogeneous quantum walk on the 1-D lattice with coin $C_x=  e^{i\theta \sigma_1}$ for all $x\in \Z$, we have: 
\begin{tcolorbox}
\vspace{-0.7em}
\begin{align}\label{eqns: Feynman Formula}
  \psi_n(x,z)&= e^{n\theta}\E\left[i^{S_n} \psi_0(X_n, Z_n)  \mid (S_0, Z_0, X_0) = (0, z, x)\right]
\end{align}
\end{tcolorbox}
\noindent for $(n,x, z) \in \mathbb N_0 \times \mathbb Z \times \{\pm 1\} $, where $(S_\cdot, Z_\cdot, X_\cdot)$ are 
in Definition \ref{defn: MA Model}. 
\end{theorem}
  \begin{proof}
  Starting with the result in Lemma \ref{lemma: Poisson-like_coefficients}, and using Poisson distributions, we have 
\begin{align*}
  \psi_n(x_0,z_0)  &= \sum_{k_1, \cdots, k_n=0}^\infty \frac{(i\theta)^{k_1+\cdots+k_n}}{k_1!\cdots k_n!} \psi_0(x_n, z_n)\nonumber\\ 
    &= e^{n\theta}\sum_{k_1, \cdots, k_n=0}^\infty (i)^{k_1+\cdots+k_n}\frac{e^{-\theta}\theta^{k_1}\cdots
    e^{-\theta}\theta^{k_n}}{k_1!\cdots k_n!} \psi_0(x_n, z_n)\nonumber\\
   &= e^{n\theta}\E\left[i^{S_n} \psi_0(X_n, Z_n)  \mid (S_0, Z_0, X_0) = (0, z, x)\right],
\end{align*}
which is formula \eqref{eqns: Feynman Formula} for $n\geq 1$ with $x_0 = x$ and $z_0 = z$.  
\end{proof}

\subsection{Monte Carlo simulations}
\label{sec: MonteCarloSimulations}

It is very tempting to use our formula \eqref{eqns: Feynman Formula} to implement a Monte Carlo computation of the wave function $\psi_n$ or its square modulus to approximate the probability of presence. By brute force, we would generate many realizations of the Poisson random variables $(N_1^{(k)},\cdots,N_n^{(k)})$ for $k=1,\cdots,M$, and then obtain
\[
\psi_n(x,z)\approx e^{n\theta}\frac{1}{M}\sum_{k=1}^M i^{S^{(k)}_n}\psi_0(X^{(k)}_n,Z^{(k)}_n),
\]
where $(S^{(k)}_n,Z^{(k)}_n,X^{(k)}_n)$ are computed from $(N_1^{(k)},\cdots,N_n^{(k)})$ according to the formulas in Definition \ref{defn: MA Model}. Well, we tried and that doesn't work well unless using a huge unrealistic number of realizations even for a relatively small time $n$. The detailed study of the sample size to compute it with stochastic simulation methods that include importance sampling, e.g., \cite{MR3784496} is an ongoing project. In fact, we are in a comparable situation with the stochastic model for the telegrapher's equation described by the great Mark Kac in \cite{Kac_1974}. 
Here is what he wrote about implementing a Monte Carlo computation for a discrete time transport system:

\begin{quote}
    ``This is a completely ridiculous scheme. Because in order to have good accuracy $\Delta t$ had better be small. I don't know how small, but certainly you ought to discretize reasonably well. But that means the probability of reversing directions is very close to zero. It is a very unlikely event. This is the type of situation which is very difficult to handle by any kind of Monte Carlo technique. Because with such a small probability you would have to have an enormous number of particles. You would need a really ridiculous number. Otherwise, the fluctuation will be enormous."
\end{quote}

In fact, for our system, due to its Schr\"odinger nature, the  situation is even worse as cancellations between small probabilities are happening and crucial in getting the right answer. This will be made very precise in Section \ref{sec: weaklimit}. 

However, in this same article, Mark Kac suggests a rescaling of the system to pass to the continuous-time limit. He obtained a probabilistic representation of the solution of a transport system of Boltzmann type and observed that in this limiting situation, Monte Carlo technique makes perfect sense and is amazingly efficient. This has been the basis and widely used technique for solving transport equations with Boltzmann interactions (see \cite{Pardoux89}).

We will perform such a rescaling of our model to pass to the continuous-time limit and identify a limiting system of PDEs of Schr\"odinger transport type for which our probabilistic representation is an efficient Monte Carlo scheme. This will be presented in Section \ref{sec: telegraph}.

%% file: Section3.tex
\section{Relation between quantum walks and spectral properties of Markov additive models}\label{sec: weaklimit}
In this section we concentrate on the case of a  quantum walk with homogeneous coin operator $C=I\otimes 
e^{i\theta\sigma_1}$. 
\subsection{Motivation - ballistic weak limit of quantum walk}
If we start the quantum walk described in Theorem \ref{thm: Feynman_Formula} with the initial state $\psi_0\in S(\ell^2(\Z\times \{\pm 1\}))$, that is $\sum_{x\in \Z}(\abs{\psi_0(x, 1)}^2 + \abs{\psi_0(x, -1)}^2) =1$, then also $\psi_n\in S(\ell^2(\Z\times \{\pm 1\}))$ since the walk operator $W = SC$ is  unitary. It follows that we can define random variables $\Xi_n $ supported on $\Z$ with probability mass functions defined by
\begin{align} \label{eq: Psin}
    \P(\Xi_n = x) = \abs{\psi_n(x, 1)}^2 + \abs{\psi_n(x, -1)}^2, \quad x \in \mathbb Z,\quad n \ge 0.
\end{align}
Observe that sampling from $\Xi_n$ is obtained by  measuring the underlying quantum system at time $n$, and therefore does not define a process in time. In other words $\Xi_n$ is only a flow of marginals.

It first appeared in Konno's work \cite{konno_2005} that if we start with the state
\begin{align}
    \psi_0(x, z):= R\I_{x=0, z=1} + L\I_{x=0, z= -1}  \quad \mbox{where}\quad (R, L)\in \C^2 \quad \mbox{with}\quad \abs{R}^2 + \abs{L}^2=1,
\end{align}
 there exists a weak limit for the sequence of random variables ${\Xi_n}/{n}$ with the limiting probability density function on the real axis given by $f_K\cdot \omega_{\psi_0}$, where
\begin{equation}\label{eqns: Konnos_distribution}
\begin{split}
    f_K(y;\theta) &:= \frac{1}{\pi} \frac{\sin\theta }{(1-y^2)\sqrt{\cos^2\theta-y^2}} \I_{\abs{y}\leq \cos\theta} ,
\\
 \omega_{\psi_0}(y) &:= 1+(\abs{R}^2 - \abs{L}^2 +2 \tan\theta \Im(R\overline{L}))y, \quad y\in \R, 
\end{split}
\end{equation}
which could be found in, for instance,  \cite{konno_2005, SunadaTate_2012} with slightly different notations. Indeed one can check that the function $f_K\cdot \omega_{\psi_0}$ is non-negative and integrates to one over $\R$.  
In order to derive this convergence, we characterize the distribution  of ${\Xi_n}/{n}$ by
\begin{align} \label{eqns: quantum_walk_cdf}
    \P \Big(y_1<\frac{\Xi_n}{n} \leq y_2\Big) 
    &= \sum_{x\in \Z,\, ny_1< x\leq ny_2} (\abs{\psi_n(x, 1)}^2 + \abs{\psi_n(x, -1)}^2)\\
    &= n\int_{\frac{1}{n}\floor{ny_1}}^{\frac{1}{n}\floor{ny_2}} (\abs{\psi_n(\ceil{an}, 1)}^2 +  \abs{\psi_n(\ceil{an}, -1)}^2) \mathrm d a,\nonumber
\end{align}
for every  $y_1<y_2\in \R$. 

In Figure \ref{fig:density} we plot on the left an approximation of this density obtained by solving recursively \eqref{eq: def2.1}, and on the right we plot the limiting density $f_K \cdot \omega_{\psi_0} \equiv f_K $ in \eqref{eqns: Konnos_distribution} with $\theta = \pi / 4$, $R = L = 1/ \sqrt{2}$.

\begin{figure}[H] 
\begin{center}
\begin{tabular}{cc}
\includegraphics[scale=.4]{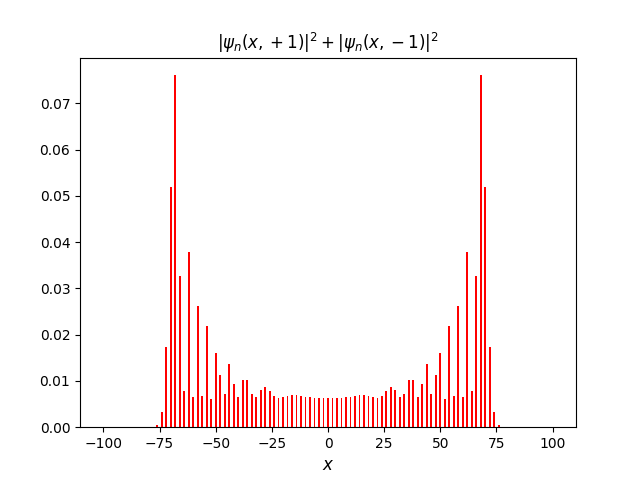} & 
\includegraphics[scale=.4]{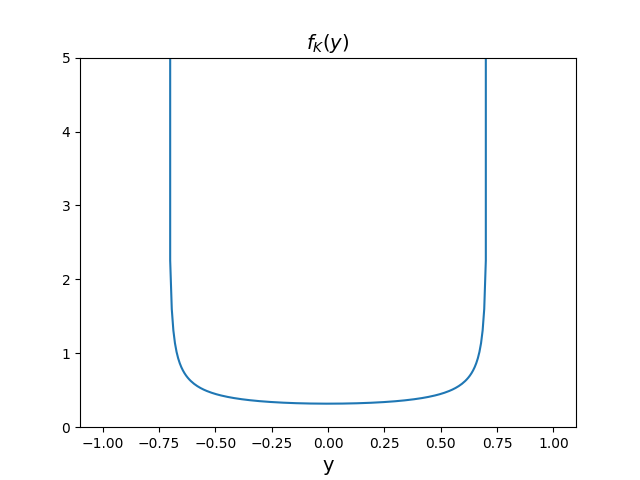}  
 \\
\end{tabular}
\caption{(Left) Probability mass function $\mathbb P ( \Xi_{100} = x) $, $-100 \le x \le 100$ of the discrete-time quantum walk computed recursively from \eqref{eq: def2.1} and then \eqref{eq: Psin} with the  balanced coin $e^{i (\pi/4)\sigma_1}$ and the initial condition described in Example \ref{ex: QWex2} and Figure \ref{fig: quantum_walk_exmple}. (Right) The limiting  probability density function $f_{K}(y; \theta ) $, $\lvert y \rvert < 1/\sqrt{2}$ in \eqref{eqns: Konnos_distribution} with $\theta = \pi / 4$, $R = L = 1/ \sqrt{2}$. \label{fig:density}
 }
\end{center}
\end{figure}

On the other hand, given the four-state Markov chain $(S_n)$ defined in Definition \ref{defn: MA Model} supported in $\Omega:=\{0, 1, 2, 3\}$, we can define $T_n:= \sum_{j=0}^{n-1} (-1)^{S_j}$ for every $n\geq 1$ and $T_0:= 0$.  

Writing the Feynman formula \eqref{eqns: Feynman Formula} using the pair $(S_n, T_ n)$ gives:
\begin{proposition}
Denoting $\psi_n^\xi$ the nth step states for the quantum walk with initial condition $\psi_0^\xi(x, z):= \I_{x=0, z = \xi}$, we have for every fix $n\in \N$:
\begin{align}\label{eqns: Feynman_Formula_Probability_difference}
     \psi_n^\xi(x, z)  =
 \begin{cases}
     e^{n\theta}[\P(T_n = xz, S_n = 0) -\P(T_n = xz, S_n = 2) ], \quad z = \xi\\
     ie^{n\theta}[\P(T_n = xz, S_n = 1) -\P(T_n = xz, S_n = 3) ], \quad z = -\xi  , 
 \end{cases}    
\end{align}
for every $x\in \Z$ and $z= \pm 1$.

\end{proposition}
\begin{proof}
Substituting the relations $X_n = x - zT_n$ and $Z_n = (-1)^{S_n}$ in the Feynman formula \eqref{eqns: Feynman Formula} gives:
\begin{align}\label{eqns: Feynman_Formula_Tn}
    \psi_n(x, z) = e^{n\theta} \E\left[ i^{S_n} \psi_0(x - zT_n, z(-1)^{S_n})\right] , \quad n \ge 0 . 
\end{align}
For the initial states  $\psi_0^\xi(x, z):= \I_{x=0, z = \xi}$ with $\xi = 1$ or $\xi = -1$, we obtain \eqref{eqns: Feynman_Formula_Probability_difference}. 
\end{proof}
By the ergodic theorem for the irreducible aperiodic finite-state Markov chain $(S_n)$ with the doubly stochastic transition probability matrix (see \eqref{eqns: transition_of_Sn} in the next section for the precise expression), 
its stationary distribution is uniform, and so we have:
\begin{align}
\lim_{n\to \infty} \frac{T_n}{n} = 0 \quad a.s. 
\end{align}
Rescaling the space  in \eqref{eqns: Feynman_Formula_Probability_difference} by setting $x = an$ for $a\in \R$ such that $an\in \Z$, by the theory of large deviations \cite{Dembo_2010}, we see that each probability term $\mathbb P ( T_n / n = a z , S_n= \cdot) $ in \eqref{eqns: Feynman_Formula_Probability_difference} is exponentially small for large $n$ and $a \neq 0$. Equation \eqref{eqns: Feynman_Formula_Probability_difference} tells us that the quantum walk amplitude $ \psi_n^\xi (an, z) $ is a re-scaled (by a factor $e^{n\theta}$) difference of small probabilities led by large deviation. 

The remainder of this section is devoted to link the asymptotic behavior of \eqref{eqns: Feynman_Formula_Probability_difference} to the properties of the Markov additive model $(S_n, T_n)$. 
\subsection{Joint distribution of \texorpdfstring{$(S_n, T_n)$}{(Sn,Tn)}}
Setting up notations, we define $p_n(s, t):= \P(S_n = s, T_n = t)$ for every $s\in\Omega$ and $t\in \Z$. We denote the transition probability kernel of $(S_n)$ by $P:= \{p(s',s)\}_{(s',s)\in \Omega^2}$. In matrix form, the kernel is given by
\begin{align}\label{eqns: transition_of_Sn}
 &   P = 
\begin{bmatrix}
    p_0 & p_1 & p_2 & p_3 \\
    p_3 & p_0 & p_1 & p_2\\
    p_2 & p_3 & p_0 & p_1 \\
    p_1 & p_2 & p_3 & p_0
\end{bmatrix},
\quad \mbox{where}\quad 
\begin{bmatrix}
    p_0\\
    p_1\\
    p_2\\
    p_3
\end{bmatrix}:=\frac{e^{-\theta}}{2}
\begin{bmatrix}
    \cosh\theta+\cos\theta\\ \sinh\theta+\sin\theta\\ \cosh\theta-\cos\theta\\ \sinh\theta-\sin\theta
\end{bmatrix}.  
\end{align}
Here, $p_s = \mathbb P(N_1 \equiv s \mod 4) $, $s \in \Omega$ can be obtained, for instance, by the probability generating function of the Poisson random variable $N_1$. We shall evaluate the characteristic functions  
\begin{align*}
    \mathbb E (e^{i k T_n}  {\I_{\{S_n = s\}}}) , \quad k \in (-\pi, \pi) ,\quad s\in \Omega, \quad n \ge 1,
\end{align*}
and invert them to obtain the joint distribution $p_n$ in \eqref{eqns: joint_marginal_Fourier_inversion}.
Indeed, it is straightforward to see that $(p_n)$ satisfies the forward dynamics
\begin{align*}
p_{n+1}(s, t) = \sum_{s'\in \Omega} p_n(s', t-f(s'))p(s', s) \quad\mbox{for} \quad n\geq 1, \quad 
    p_{0}(s, t) = \I_{s=0, t = 0},
\end{align*}
for all $(s, t)\in \Omega\times \Z$. Taking the Fourier transform on the $t$ component and defining the respective Fourier transform $\hat p(s, k):= \sum_{t\in \Z} p(s, t)e^{ikt}$ for every $k\in (-\pi, \pi)$, we have 

\begin{align}\label{eqns: Fourier_dynamics_of_joint}
    \hat p_{n+1}(s, k) = \sum_{s'\in \Omega} \hat p_n(s', k) e^{ikf(s')} p(s', s) \quad\mbox{for} \quad n\geq 1, \quad  
    \hat p_0 (s, k) &= \I_{s=0}, 
\end{align}
for every $(s, k)\in \Omega\times (-\pi, \pi)$. Defining the tilted kernels $P_k:= \{p(s', s)e^{ikf(s')}\}_{(s', s)\in \Omega^2}$ with  parameters $k\in (-\pi,\pi)$,  we can re-write  the Fourier dynamics \eqref{eqns: Fourier_dynamics_of_joint} to be the matrix dynamics 
\begin{align*}
    \hat p_{n+1}(k) = P_k \hat p_n(k)\quad\mbox{for} \quad n\geq 1, \quad  \hat{p}_0(k) = 
\begin{bmatrix}
    1 & 0 & 0 & 0
\end{bmatrix}^T,
\end{align*}
where $\hat p_{n}(k):= \{\hat p_n(s,k)\}_{s\in \Omega}$ are vectors and $P_k$ has the matrix representation
\begin{align}\label{eqns: tilted_kernel_matrix}
    P_k = \begin{bmatrix}
    e^{ik}p_0 & e^{ik}p_1 & e^{ik}p_2 & e^{ik}p_3 \\
    e^{-ik}p_3 & e^{-ik}p_0 &e^{-ik} p_1 & e^{-ik}p_2\\
    e^{ik}p_2 & e^{ik}p_3 & e^{ik}p_0 & e^{ik}p_1 \\
    e^{-ik}p_1 &e^{-ik} p_2 & e^{-ik}p_3 & e^{-ik}p_0
\end{bmatrix}.
\end{align}
By the inverse Fourier transform we can thus obtain the joint marginals of $(S_n, T_n)$ to be 
\begin{align}\label{eqns: joint_marginal_Fourier_inversion}
    p_n(s, t) = \frac{1}{2\pi}\int_{-\pi}^\pi (P_k^n \hat p_0(k))(s) e^{-ikt}\mathrm d k.
\end{align}
Subsequent simplifications can be made by studying the spectrum of the matrices $P_k$:
\begin{proposition}\label{prop: split_spectrum}
   The subspaces of $\ell_2(\Omega)$ defined by $A:= span \{e_0+e_2, e_1+e_3\}$ and $B:= span \{e_0 - e_2, e_1 - e_3\}$ are invariant subspaces of the tilted kernel $P_k$. Furthermore the matrix representations of $P_k$ under $A, B$ are given respectively by the diagonalizable matrices 
\begin{align}\label{eqns: split_matrix}
   A(k) := e^{-\theta}\begin{bmatrix}
    e^{ik}\cosh\theta & e^{ik}\sinh\theta\\
    e^{-ik}\sinh\theta & e^{-ik}\cosh\theta
\end{bmatrix}, \quad     B(k) := 
e^{-\theta}\begin{bmatrix}
    e^{ik}\cos\theta & e^{ik}\sin\theta\\
    -e^{-ik}\sin\theta & e^{-ik}\cos\theta
\end{bmatrix},
\end{align}
with the eigenvalues of $A(k)$ being 
\begin{align*}
    \lambda_{A,1}(k) &=e^{-\theta}[ \cos k\cosh\theta + \sqrt{\cos^2k\cosh^2\theta-1} ], \quad \mbox{and}\\
      \lambda_{A,2}(k) &=e^{-\theta}[ \cos k\cosh\theta - \sqrt{\cos^2k\cosh^2\theta-1} ],
\end{align*}
and the eigenvalues of $B(k)$ being
\begin{align*}
    \lambda_{B,1}(k) &=e^{-\theta}[ \cos k\cos\theta + \sqrt{\cos^2k\cos^2\theta-1} ], \quad \mbox{and}\\
      \lambda_{B,2}(k) &=e^{-\theta}[ \cos k\cos\theta - \sqrt{\cos^2k\cos^2\theta-1} ],
\end{align*}
where we specify the square root function as 
\begin{align*}
    \sqrt{c} = 
\begin{cases}
    \sqrt{c}, & \quad c\geq 0\\
    i\sqrt{-c}, & \quad c< 0
\end{cases}.
\end{align*}
\end{proposition}
\begin{proof}
  It can be easily checked by direct computation.  
\end{proof}

Proposition $\ref{prop: split_spectrum}$ tells us that $P_k$ is diagonalizable with eigenvalues $\lambda_{A, j}(k)$ and $\lambda_{B, j}(k)$ for $j = 1, 2$ . Regarding the left eigenvectors of $P_k$, denoted by $\ell_{A,j}(k):= \{\ell_{A,j}(k, s)\}_{s\in \Omega}$ and $\ell_{B,j}(k):= \{\ell_{A,j}(k, s)\}_{s\in \Omega}$; and right eigenvectors of $P_k$, denoted by $r_{A_j}(k):= \{r_{A,j}(k, s)\}_{s\in \Omega}$ and $ r_{B_j}(k):= \{r_{B,j}(k, s)\}_{s\in \Omega}$,   corresponding to the respective eigenvalues $\lambda_{A,j}(k), \lambda_{B,j}(k)$, we can now further decompose \eqref{eqns: joint_marginal_Fourier_inversion} 
 into the sum of four integrals by spectral decomposition:
\begin{align}\label{eqns: marginals_as_integrals_0}
    p_n(s, t) &= \sum_{j=1}^{2} \frac{1}{2\pi}\int_{-\pi}^\pi r_{A,j}(k, 1)\ell_{A,j}(k, s) \lambda_{A,j}^n(k) e^{-ikt}\mathrm d k\\
    &\hskip 1cm +\sum_{j=1}^{2} \frac{1}{2\pi}\int_{-\pi}^\pi r_{B,j}(k, 1)\ell_{B,j}(k, s) \lambda_{B,j}^n(k) e^{-ikt} \mathrm d k,\nonumber\\
     &= \sum_{j = 1, 2}(I_n(s, t; A_j) + I_n(s, t; B_j)) \nonumber
\end{align}
for every $s\in \Omega$ and $t\in \Z$,
where we define for $j=1,2$:
\begin{align}\label{eqns: marginals_as_integralsA}
I_n(s, t; A_j)&=\frac{1}{2\pi}\int_{-\pi}^\pi r_{A,j}(k, 1)\ell_{A,j}(k, s) \lambda_{A,j}^n(k) e^{-ikt}\mathrm d k,\\
I_n(s, t; B_j)&=\frac{1}{2\pi}\int_{-\pi}^\pi r_{B,j}(k, 1)\ell_{B,j}(k, s) \lambda_{B,j}^n(k) e^{-ikt} \mathrm d k.
\label{eqns: marginals_as_integralsB}
\end{align}

\subsection{Derivation of asymptotics via spectral properties}

\subsubsection{Quantum walk amplitudes}

Since \eqref{eqns: Feynman_Formula_Probability_difference} expresses amplitude of a quantum walk as the difference of joint probability of the pair $(S_n, T_n)$,  we substitute \eqref{eqns: marginals_as_integrals_0} into formula \eqref{eqns: Feynman_Formula_Probability_difference} to obtain the key observation:

\begin{theorem}\label{thm: absence_of_A_j}
The amplitude $\psi_n^\xi (x, z)$ of the quantum walk starting at the state $\psi_0^\xi(x, z):= \I_{x = 0, z = \xi}$ is given by 
\begin{align}\label{eqns: absence_of_A_j}
\psi_n^\xi(x, z) =
\begin{cases}
   2e^{n\theta}\sum_{j=1}^2 I_n(0, xz; B_j), \quad z = \xi, \\
  2ie^{n\theta} \sum_{j=1}^2  I_n(1, xz; B_j), \quad z = -\xi,
\end{cases}
\end{align} 
where the effects of eigenvalues and eigenvectors from $A(k)$ have been cancelled.  
\end{theorem}

\begin{proof}
A direct substitution of \eqref{eqns: marginals_as_integrals_0} into formula \eqref{eqns: Feynman_Formula_Probability_difference}  gives 
for $\xi = \pm 1$ that:
\begin{align*}
    \psi^\xi_n(x, z) =
\begin{cases}
    e^{n\theta} \sum_{j=1}^2 \left[I_n(0,xz; A_j)+I_n(0, xz; B_j)-I_n(2, xz; A_j)-I_n(2, xz; B_j)\right], \quad z= \xi, \\
     ie^{n\theta} \sum_{j=1}^2\left[ I_n(1,xz; A_j)+I_n(1, xz; B_j)-I_n(3, xz; A_j)-I_n(3, xz; B_j)\right], \,\, z = -\xi
\end{cases} 
\label{eqns:Feynman_difference_of_integrals}
\end{align*}
for every $x\in \Z$ and $z = \pm 1$. By Proposition \ref{prop: split_spectrum}, the left eigenvectors satisfy
\begin{align*}
\ell_{A, j}(k, 0) = \ell_{A, j}(k, 2), \quad \ell_{A, j}(k, 1) = \ell_{A, j}(k, 3),\\
    \ell_{B, j}(k, 0)  = -\ell_{B, j}(k, 2), \quad \ell_{B, j}(k, 1) = -\ell_{B, j}(k, 3)
\end{align*}
for $j = 1, 2$. Therefore, by the definition of the integrals in \eqref{eqns: marginals_as_integralsA} and \eqref{eqns: marginals_as_integralsB}, we have:
\begin{align*}
    I_n(0, t; A_j) = I_n(2, t; A_j),  \quad I_n(1, t; A_j) = I_n(3, t; A_j),\\ 
     I_n(0, t; B_j) = -I_n(2, t; B_j),  \quad I_n(1, t; B_j) = -I_n(3, t; B_j)
\end{align*}
for every $t\in \Z$ and $j = 1, 2$. Then \eqref{eqns: absence_of_A_j} follows. 
\end{proof}

\begin{remark}
    In fact, the matrices $A(k)$ are driving the large deviation regime of the Markov additive model $(S_n, T_n)$. Because of the cancellation in the difference of small probabilities as seen in \eqref{eqns: Feynman_Formula_Probability_difference}, the only factors that drive the asymptotic behavior of $\psi_n^\xi$ are the matrices $B(k)$. 
\end{remark}

\begin{remark}\label{remark: Relation_with_Grimmet}
In Proposition \ref{prop: split_spectrum}, using the alternative basis $\{e_0 - e_2, i(e_1- e_3)\}$ for the subspace $B$, we can rewrite the matrix representation of $P_k$ over it as
\begin{align*}
    e^{-\theta}
\begin{bmatrix}
e^{ik} & \\
 & e^{-ik}
\end{bmatrix}
\begin{bmatrix}
    \cos \theta & i\sin \theta \\
    i\sin \theta & \cos \theta 
\end{bmatrix} = e^{-\theta}\diag(e^{ik}. e^{-ik})e^{i\theta \sigma_1},
\end{align*}
where $\diag (e^{ik}, e^{-ik})e^{i\theta \sigma_1}$ is the matrix representation of the quantum walk operator $W:=SC$ over the Fourier domain in space, which is first documented in \cite{Grimmett_2004}. One can easily show that \eqref{eqns: absence_of_A_j} in Theorem \ref{thm: absence_of_A_j} still holds using this matrix representation, which coincides with the integral studied in Equation (1.20) in \cite{SunadaTate_2012}. In the following, we continue with our initial natural choice of basis in Proposition \ref{prop: split_spectrum}. 
\end{remark}
\subsubsection{Deriving quantum walks asymptotics}
With the initial conditions
\begin{align*}
    \psi_0:= R \psi_0^+ + L \psi_0^- \quad &\mbox{where}\quad \psi_0^{\xi}(x, z) = \I_{x=0, z=\xi}\quad \mbox{with} \quad \xi = \pm 1 ,\\
    &\mbox{and}\quad  (R, L)\in \C^2 \quad\mbox{with}\quad \abs{R}^2 + \abs{L}^2 = 1,\nonumber
\end{align*}
we can rewrite \eqref{eqns: quantum_walk_cdf} as 
\begin{equation}\label{eqns: decomposing_cdf}
\begin{split}
    \P \Big(y_1<\frac{\Xi_n}{n}\leq y_2\Big) &= n\int_{\frac{1}{n}\floor{ny_1}}^{\frac{1}{n}\floor{ny_2}}  \sum_{z=\pm 1}\abs{R\psi_n^+(\ceil{an}, z)+L\psi_n^-(\ceil{an}, z)}^2 \mathrm d a , \\ 
    &= n\int_{\frac{1}{n}\floor{ny_1}}^{\frac{1}{n}\floor{ny_2}}
    \sum_{z=\pm 1}\Big[ \abs{R}^2 \abs{\psi_n^+(\ceil{an}, z)}^2  + \abs{L}^2 \abs{\psi_n^-(\ceil{an}, z)}^2 \\
    & \qquad \qquad \qquad \qquad \qquad + 2\Re(R\overline{L} \psi_n^+(\floor{an}, z) \overline{\psi_n^-(\floor{an}, z)})\Big]\mathrm d a. 
    \end{split}
\end{equation}

In order to handle the terms in \eqref{eqns: decomposing_cdf}, we see from \eqref{eqns: absence_of_A_j} that we need
the asymptotic behavior of integrals of
\begin{align}
e^{2n\theta}\abs{I_n(s, \ceil{an}, B_1)+I_n(s, \ceil{an}, B_2)}^2,\quad s = 0,1,
\end{align}
as well as of the cross term:
\begin{align}
    e^{2n\theta}(I_n(0, \ceil{an}, B_1)+I_n(0, \ceil{an}, B_2)) \overline{(I_n(1, \ceil{an}, B_1)+I_n(1, \ceil{an}, B_2))}.
\end{align}
We have the following convergence results: 
\begin{lemma}\label{lemma: integral_asymptotics}
Let $\epsilon$ be given such that $0<\epsilon<\frac{1}{2}\cos\theta$ where we suppose w.l.o.g. that $\cos\theta>0$. For $(y_1,y_2)$ such that
$-\cos\theta+\epsilon\leq y_1<y_2\leq \cos\theta -\epsilon$, we have:
\begin{align}
 & \lim_n  n\int_{\frac{1}{n}\floor{ny_1}}^{\frac{1}{n}\floor{ny_2}}e^{2n\theta}\abs{\sum_{j=1, 2} I_n(0, \ceil{an}, B_j)}^2 \mathrm d a = \frac{1}{8}\int_{y_1}^{y_2} (1+a)^2 f_K(a; \theta) \mathrm d a,\label{eqns: integral_quantum_limit_0}\\
   & \lim_n  n\int_{\frac{1}{n}\floor{ny_1}}^{\frac{1}{n}\floor{ny_2}} e^{2n\theta}\abs{ \sum_{j=1, 2} I_n(1, \ceil{an}, B_j)}^2 \mathrm d a = \frac{1}{8}\int_{y_1}^{y_2} (1-a^2) f_K(a; \theta) \mathrm d a,\label{eqns: integral_quantum_limit_1}\\
   & \lim_n  n\int_{\frac{1}{n}\floor{ny_1}}^{\frac{1}{n}\floor{ny_2}}  e^{2n\theta}\left(\sum_{j=1, 2} I_n(0, \ceil{an}, B_j))(\overline{\sum_{j=1, 2} I_n(1, \ceil{an}, B_j)}\right) \mathrm d a  \label{eqns: integral_quantum_limit_cross} \\ \notag 
    &= \frac{1}{8}\int_{y_1}^{y_2} a(1+a)\tan\theta f_K(a; \theta )\mathrm d a,
\end{align}
where $f_K$ is the ballistic limiting density function of quantum walks introduced in \eqref{eqns: Konnos_distribution}.
\end{lemma}
\begin{proof}
    This is a consequence of the stationary phase asymptotics uniformly in $a\in [y_1,y_2]$. Note that for fixed $\epsilon>0$, we can choose $n$ large enough such that   the inequality $-\cos\theta+\epsilon/2\leq  \frac{1}{n}\floor{ny_1} \leq \frac{1}{n}\floor{ny_2}\leq \cos\theta -\epsilon/2$ holds.   
    See \hyperref[appendix: fourier]{Appendix} for the details.
\end{proof}

Coming back to the distribution of ${\Xi_n}/{n}$ given by \eqref{eqns: quantum_walk_cdf}, we deduce:
\begin{theorem} \label{thm: ballistic_limit_for_rotation}
With the initial state, 
\begin{align*}
    \psi_0(x, z):= R\I_{x=0, z=1} + L\I_{x=0, z= -1}  \quad \mbox{where}\quad (R, L)\in \C^2 \quad \mbox{with}\quad \abs{R}^2 + \abs{L}^2=1,
\end{align*}
$\frac{1}{n}\Xi_n$ converges in distribution to $Z$ where $Z$ is a real-valued random variable with probability density defined by $w_{\psi_0}(y)f_K(y; \theta )$ in \eqref{eqns: Konnos_distribution}.
\end{theorem}
\begin{proof}
We observe from \eqref{eqns: absence_of_A_j} that:
\begin{equation}
    \sum_{z=\pm 1} \abs{\psi_n^+(\ceil{an}, z)}^2 = 4e^{2n\theta} \Big( \Big\lvert \sum_{j=1}^2 I_n(0, \ceil{an}, B_j) \Big \rvert^2 +   \Big \lvert \sum_{j=1}^2 I_n(1,\ceil{-an}, B_j ) \Big \rvert^2\Big), \label{eqns: sum_simplifty_0}
    \end{equation}
    \begin{equation}
     \quad \sum_{z=\pm 1} \abs{\psi_n^-(\ceil{an}, z)}^2 = 4e^{2n\theta} \Big(\Big \lvert \sum_{j=1}^2 I_n(0, \ceil{-an}, B_j) \Big \rvert^2 +   \Big \lvert \sum_{j=1}^2 I_n(1,\ceil{an}, B_j )\Big \rvert^2\Big), \label{eqns: sum_simplifty_1}\end{equation}
     \begin{equation}\label{eqns: sum_simplifty_cross}
     \begin{split} 
     \sum_{z=\pm 1} \psi_n^+(\ceil{an}, z)\overline{\psi_n^-(\ceil{an}, z)}&=4e^{2n\theta}\Big [\sum_{j=1}^2 I_n(0, \ceil{an},B_j)\overline{\sum_{j=1}^2iI_n(1, \ceil{an}, B_j)} \\ 
     & \qquad   {}+ i\sum_{j=1}^2 I_n(1, -\ceil{an}, B_j) \overline{\sum_{j=1}^2 I_n(0, -\ceil{an}, B_j)} \Big]. 
     \end{split}
\end{equation}
Using the limits in Lemma \ref{lemma: integral_asymptotics}, we obtain: 
\begin{align} \label{eqns: quantum_sum_limit_0}
&\lim_n n\int_{\frac{1}{n}\floor{n y_1}}^{\frac{1}{n}\floor{n y_2}} \sum_{z=\pm 1}\abs{\psi_n^+ (\ceil{an}, z)}^2\\
&= \frac{4}{8}\int_{y_1}^{y_2} f_K(a; \theta)[(1+a)^2+(1-a^2)] \mathrm d a = \int_{y_1}^{y_2} (1+a)f_K(a; \theta) \mathrm d a,\nonumber\\ \label{eqns: quantum_sum_limit_1}
    &\lim_n n\int_{\frac{1}{n}\floor{n y_1}}^{\frac{1}{n}\floor{n y_2}} \sum_{z=\pm 1}\abs{\psi_n^- (\ceil{an}, z)}^2\\
    &= \frac{4}{8}\int_{y_1}^{y_2} f_K(a; \theta)[(1-a^2)+(1-a)^2]da = \int_{y_1}^{y_2} (1-a)f_K(a; \theta) \mathrm d a,\nonumber\\ 
    \label{eqns: quantum_sum_limit_cross}
   & \lim_n n\int_{\frac{1}{n}\floor{n y_1}}^{\frac{1}{n}\floor{n y_2}} \sum_{z=\pm 1}\psi_n^+ (\ceil{an}, z)\overline{\psi_n^-(\ceil{an}, z)} \\
   &= \frac{4}{8}\int_{y_1}^{y_2} f_K(a; \theta)[-ia(1+a)\tan\theta + i(-a)(1-a)\tan \theta ] \mathrm d a \nonumber\\
    &=-i\int_{y_1}^{y_2} af_K(a; \theta) \mathrm d a\nonumber, 
\end{align}
where we  replaced $a$ with $-a$ in the limits derived in Lemma \ref{lemma: integral_asymptotics} to deal with limits of integrals with the parameter $-\ceil{an}$ instead of $\ceil{an}$ in \eqref{eqns: sum_simplifty_0}, \eqref{eqns: sum_simplifty_1} and \eqref{eqns: sum_simplifty_cross} by making a change of variable in the outer integrals  that concerns the parameter $a$ in Lemma \ref{lemma: integral_asymptotics}. 

Applying these limits \eqref{eqns: quantum_sum_limit_0}, \eqref{eqns: quantum_sum_limit_1}, and \eqref{eqns: quantum_sum_limit_cross} to \eqref{eqns: decomposing_cdf}, we deduce 
\begin{align}
&\lim_n \P \Big(y_1<\frac{\Xi_n}{n}\leq y_2\Big)\\
&=\int_{y_1}^{y_2} f_K(a; \theta) \Big[\abs{R}^2 (1+a) + \abs{L}^2 (1-a) + 2\Re \Big(\frac{R\overline{L}\tan\theta a}{i}\Big)\Big] \mathrm d a \notag \\
&=\int_{y_1}^{y_2} f_K(a; \theta)[1+(\abs{R}^2-\abs{L}^2 + 2\tan\theta \Im (R\overline{L}))a] \mathrm d a, \notag 
\end{align}
which is precisely the integral over $[y_1,y_2]$ of the density $f_K\cdot \omega_{\psi_0}$ defined in \eqref{eqns: Konnos_distribution}. The last step consists in sending $\epsilon$ to $0$ and using 
\begin{align}
  \int_{-\cos\theta}^{\cos\theta}f_K(a;\theta) \omega_{\psi_0}(a) \mathrm d a=1,
\end{align}
to conclude that the density is $f_K\cdot\omega_{\psi_0}$ over the full real line.
\end{proof}

%% file: Section4.tex
\section{Hyperbolic scaling limit and relation with the telegraph process}
\label{sec: telegraph}
We consider the case with the homogeneous coin $C = \id \ox e^{i\theta\sigma_1}$ studied in Section \ref{sec: Feynman_Formula}. Our goal is to rescale the discrete time $n\in \N$ and discrete space $x\in \Z$ for $\psi_n(x, z)$ in order to obtain a limiting PDE  in continuous time $t>0$ and space $y\in \R$. This will be achieved by using the Feynman formula \eqref{eqns: Feynman Formula} with the processes $(S_n,Z_n,X_n)$ introduced in Definition \ref{defn: MA Model}, whose scaling limit is given by

\begin{lemma}\label{lemma: classical_scaling_limit}
For all $\epsilon>0$, denoting by $N_n^\epsilon, n=1,2\cdots$ our i.i.d. Poisson random variables with parameter $\epsilon\lambda>0$, and $S_n^\epsilon=\sum_{j=1}^n N_j^\epsilon$ with $S_0^\epsilon=0$, we have the following weak convergence in the space of cadlag paths $D(0, \infty)$:
\begin{align}\label{cvgD1}
   \lim_{\epsilon \to 0}\left(S^\epsilon_{\lfloor t/\epsilon\rfloor}\right)_{t\geq 0} \overset{\L}{=} \left(N_t\right)_{t\geq 0}, \quad \mbox{a Poisson process of intensity}\quad \lambda>0.
\end{align}
Consequently, denoting $Z_n^\epsilon=(-1)^{S_n^\epsilon}z$,
\begin{align}\label{cvgD2}
   \lim_{\epsilon \to 0} \left(Z^\epsilon_{\lfloor t/\epsilon\rfloor}\right)_{t\geq 0} \overset{\L}{=} \left((-1)^{N_t}z\right)_{t\geq 0},\quad\mbox{denoted by}\quad (Z_t)_{t\geq 0}
\end{align}Denoting $X^\epsilon_n=\floor{y/\epsilon}-\sum_{j=0}^{n-1} (-1)^{S_n^\epsilon}$ for $y\in \R$, we have also
\begin{align}\label{cvgD3}
   \lim_{\epsilon \to 0}\left(\epsilon X^\epsilon_{\lfloor t/\epsilon\rfloor}\right)_{t\geq 0} \overset{\L}{=} \left(y-\int_0^t (-1)^{N_s} \mathrm d s\right)_{t\geq 0},\quad\mbox{denoted by}\quad (Y_t)_{t\geq 0}.
\end{align}
Here the process $T_t:= \int_0^t (-1)^{N_s} \mathrm d s $ is the well-known telegraph process driven by the direction process $D_t:= (-1)^{N_t}$; and we can write $Z_t = zD_t$ and $Y_t = y - T_t$. 
\end{lemma}

\begin{proof}
For \eqref{cvgD1}, the limiting distribution of an increment $S^\epsilon_{\lfloor t/\epsilon\rfloor}-S^\epsilon_{\lfloor s/\epsilon\rfloor}$ follows by using that the sum of independent Poisson random variables is a Poisson random variable with parameter the sum of the parameters and 
\[
\lim_{\epsilon\to 0}(\lfloor t/\epsilon\rfloor-\lfloor s/\epsilon\rfloor)\epsilon \lambda=\lambda (t-s).
\]
The independence of increments follows from the independence of the random variable $N_n$'s.

Convergence \eqref{cvgD2} follows from the convergence in distribution of integer-valued random variables.

For convergence \eqref{cvgD3} one can introduce the times of increase $n_j^{\epsilon}$ of $S^\epsilon_{\lfloor t/\epsilon\rfloor}$ noting that only the increments by $1$ matter as the increments by more than $1$ are of higher order in $\epsilon$. It follows that $\lim_{\epsilon\to 0}\epsilon n_j^{\epsilon}= T_j$, the $j$'th time of increase of $N_t$. 

The convergence \eqref{cvgD3} follows from $\lim_{\epsilon \to 0}\epsilon \lfloor y/{\epsilon}\rfloor=y$ and the decomposition of the sum appearing in $\epsilon X^\epsilon_{\lfloor t/\epsilon\rfloor}$ over the intervals $[n_{j}^{\epsilon},n_{j+1}^{\epsilon}[$ to obtain the convergence to the integral appearing in $Y_t$ over the intervals $[T_j,T_{j+1}[$.
\end{proof}

Lemma \ref{lemma: classical_scaling_limit} motivates us to consider, in \eqref{eqns: Feynman Formula},   the rescaling of space $x\in \Z$ with $\floor{y/\epsilon}$, time $n\in \N$ with $\floor{n/\epsilon}$ and the coin parameter $\theta>0$ with $\epsilon \lambda$ where $\lambda>0$. We also  rescale the initial state $\psi_0$ as follows:

\begin{theorem}\label{thm: quantum_scaling_limit}
Let $f:\R\times \{\pm 1\}\to \C$ be a function which is $C_0$ in the first entry, that is, $\lim_{y\to \pm \infty}\abs{f(y, \pm 1)} = 0$.  For every $\epsilon>0$, we define the state vector $\psi_0^\epsilon\in \ell^2(\Z\times \{\pm 1\})$ (not necessarily normalized) by
\begin{align}
    \psi_0^\epsilon(x, z) =  f(\epsilon x, z) 
\end{align}
and let $\psi_n^\epsilon$ be the n-th step state of this initial condition with the homogeneous coin $e^{i\epsilon \lambda \sigma_1}$. Then for every $t>0, y\in \R$ and $z = \pm 1$, the point-wise limit
\begin{align}\label{eqns: scaling_quantum_expectation}
   \psi(t, y, z):= \lim_{\epsilon\to 0}  \psi_{[t/\epsilon]}^\epsilon ([y/\epsilon], z) = e^{t\lambda}\E_{}(i^{N_t} f(Y_t, Z_t)\mid (N_0, Y_0, Z_0)=(0, y, z))
\end{align}
exists.
\end{theorem}
\begin{proof}
    By the Feynman formula \eqref{eqns: Feynman Formula}, we have 
\begin{align*}
    \psi_{[t/\epsilon]}^\epsilon ([y/\epsilon], z) &= e^{[t/\epsilon]\cdot \epsilon \lambda} \E\left[i^{S_{[t/\epsilon]}^\epsilon} \psi_0^\epsilon([y/\epsilon]- z T_{[t/\epsilon]}^\epsilon, z(-1)^{S_{[t/\epsilon]}^\epsilon}) \right]\\
    &= e^{[t/\epsilon]\cdot \epsilon \lambda} \E\left[i^{S_{[t/\epsilon]}^\epsilon} f(\epsilon[y/\epsilon]- z\epsilon T_{[t/\epsilon]}^\epsilon, z(-1)^{S_{[t/\epsilon]}^\epsilon})\right]. 
\end{align*}
Together with the weak convergence result in Lemma \ref{lemma: classical_scaling_limit}, we have
\begin{align*}
    \psi(t, y, z) &= \lim_{\epsilon\to 0} \psi_{[t/\epsilon]}^\epsilon ([y/\epsilon], z) = e^{t\lambda}\E\left[i^{N_t} f(y - z T_t, z(-1)^{N_t})\right] \\
    &= e^{t\lambda} \E\left[i^{N_t} f(Y_t, Z_t)\mid Y_0 = y, Z_0 = z)\right] . 
\end{align*}
\end{proof}

It is now very natural to write a system of quantum transport PDEs solved by $\psi$, for which we have the probabilistic representation \eqref{eqns: scaling_quantum_expectation}. In fact, this type of PDEs, known as Dirac PDEs, has been studied in the Physics literature and we refer to the papers \cite{MS4-2018} and \cite{MS-2020} where nonlinear PDEs are also discussed.

\begin{proposition} \label{prop: 4.3}
   The function $\psi(t, y, z)$ defined in Theorem \ref{thm: quantum_scaling_limit} is a solution to the linear system of PDEs 
\begin{align}\label{PDE1}
\left\{
\begin{array}{ccc}
    \partial_t \psi_+&=&-\partial_y \psi_++i\lambda \psi_-, \\
    \partial_t \psi_-&=&\partial_y \psi_-+i\lambda\psi_+
    \end{array}
    \right.
\end{align}
with $\psi(0,y,z)=f(y,z)$ and where $\psi_{\pm}(t,y):=\psi(t,y,\pm 1)$. Additionally, we have the conservation property: 
\begin{equation}\label{conservationlaw}
\int_{-\infty}^\infty(|\psi_+(t,y)|^2+|\psi_-(t,y)|^2) \mathrm d y=1.
\end{equation}
for every $t>0$ whenever $\int_{-\infty}^\infty \abs{f(y, 1)}^2 + \abs{f(y, -1)}^2\mathrm d y  = 1$. 
\end{proposition}

\begin{proof}
 Let $h>0$ and denoting by $\E_{0, y, z}$ the expectation over the initial condition $(N_0, Y_0, Z_0) = (0, y, z)$, we have 
\begin{align}
\label{eqns: pde_proof_1} 
     & \qquad \qquad \psi(t+h, y, z) 
    = e^{(t+h)\lambda}\E_{0, y, z}\Big[ i^{N_{t+h}} f\Big(y - z\int_0^{t+h}(-1)^{N_s}\mathrm d s, z(-1)^{N_{t+h}} \Big) \Big]\\ 
    &= e^{(t+h)\lambda}\E_{0, y, z}\Big[i^{N_h} \E\Big[i^{N_{t+h} - N_h} f\Big(Y_h- Z_h\int_h^{t+h} \!\!\! (-1)^{N_s - N_h}\mathrm d s, Z_h(-1)^{N_{t+h} - N_h}\Big) \Big \vert  \F_h \Big]\Big],\notag
\end{align}
where $\F_h:= \sigma(N_s,\,s\leq h)$.  Using the fact that $(M_t):=\{N_{t+h} - N_h\}_{t\geq 0}$ is a Poisson process with parameter $\lambda$ and initial condition $M_0 = 0$, and is independent of $\F_h$, we can compute: 
\begin{align}\label{eqns: pde_proof_2}
    &\E\left[i^{N_{t+h} - N_h} f\Big(Y_h - Z_h\int_h^{t+h} (-1)^{N_s - N_h}\mathrm d s, Z_h(-1)^{N_{t+h} - N_h}\Big)\Big \vert  \F_h\right]\\ \notag 
    &= \E\left[i^{M_t} f\Big(Y_h - Z_h\int_h^{t+h} (-1)^{M_{s-h}}\mathrm d s, Z_h(-1)^{M_s} \Big)\right] = e^{-t\lambda}\psi(t, Y_h, Z_h). \notag
\end{align}
Substituting \eqref{eqns: pde_proof_2} into  \eqref{eqns: pde_proof_1} gives 
\begin{equation*}\label{eqns:pdelemma}
\begin{split}
     \psi(t+h, y, z) &= e^{h\lambda}\E_{0,y,z}(i^{N_h} \psi(t, Y_h, Z_h))\\ \notag 
    & =  e^{h\lambda}\P(N_h=0)\E_{0, y, z}(i^{N_h} \psi (t, Y_h, Z_h)\mid N_h = 0) \\
    & \qquad {} + e^{h\lambda}\P(N_h=1)\E_{0, y, z}(i^{N_h} \psi(t, Y_h, Z_h)\mid N_h = 1) + O(h^2)\\ \notag 
    &= \psi(t, y-hz, z) + ih\lambda \E_{0, y, z}\Big[\psi \Big(t, y - \int_0^h(-1)^{N_s}\mathrm d s, -z \Big) \Big \vert N_h =1 \Big] + O(h^2).
    \end{split}
\end{equation*}
Using the fact that the first jump time $T$ of $(N_t)$ follows the conditional distribution $T\mid N_h = 1\sim U[0, h]$, which is the uniform distribution over $[0, h]$, we have:
\begin{equation} 
\begin{split} \label{eqns: pde_proof_3}
    & \qquad \E_{0, y, z}\Big[\psi(t, y - z\int_0^h(-1)^{N_s}\mathrm d s, -z)\mid N_h=1\Big] \\
    & = \E_{0, y, z}[\psi(t, y-z(T - (h-T)), -z)\mid T\sim U[0, h]]
    \\
    &=\frac{1}{h}\int_0^h \psi(t, y+zh-2zs, -z) \mathrm d s.  
\end{split}
\end{equation}
Substituting \eqref{eqns: pde_proof_3} into  \eqref{eqns: pde_proof_2}, we have:
\begin{align}
    \psi(t+h, y, z) = \psi(t, y-hz, z) + i\lambda \int_0^h \psi(t, y+zh-2zs, -z)\mathrm d s + O(h^2).
\end{align}

Subtracting $\psi(t, y, z)$ on both sides, dividing by $h$, and sending $h\to 0$ gives
\begin{align}
    \p_t \psi(t, y, z) = -z\p_y\psi (t, y, z) + i\lambda \psi(t, y, -z),
\end{align}
for all $t>0, y\in \R$ and $z= \pm 1$, which is \eqref{PDE1}.

Regarding the conservation property \eqref{conservationlaw},   one first observes that 
\begin{equation}\label{limits}
\lim_{\abs{y}\to\infty}\psi_{+}(t,y)=\lim_{\abs{y}\to \infty}\psi_{-}(t, y) = 0,
\end{equation}
which follows from the probabilistic representation \eqref{eqns: scaling_quantum_expectation} by the assumption $f\in C_0$ and the fact that $\lim_{y\to \pm \infty}Y_t = \pm \infty$  almost surely using the definition \eqref{cvgD3} of $Y_t$.  Then by a direct computation using \eqref{PDE1}, one gets
\begin{align}
\partial_t(|\psi_+|^2+|\psi_-|^2)
&=\partial_y(|\psi_-|^2-|\psi_+|^2), 
\\
\mbox{which implies}\quad \p_t \int_{-\infty}^\infty (|\psi_+|^2+|\psi_-|^2)dy&=\int_{-\infty}^\infty \partial_y(|\psi_-|^2-|\psi_+|^2)dy. \notag 
    \end{align}
Using integration by part and \eqref{limits} gives that 
  $\int_{-\infty}^\infty(|\psi_+(t,y)|^2+|\psi_-(t,y)|^2)\mathrm d y$ is constant in time. Therefore, \eqref{conservationlaw} follows from $\int_{-\infty}^\infty \abs{f(y, 1)}^2 + \abs{f(y, -1)}^2 \mathrm d y  = 1$ at time $t=0$.
\end{proof}

\begin{remark}
If  we define $u(t,y,z)=e^{-i\lambda t}\psi(t,y,z)$ by removing a global phase $e^{i\lambda t}$, then we can obtain that: 
\begin{align}
    u(t,y,z)=e^{\lambda(1-i) t}\E[i^{N_t} f(Y_t, Z_t)\mid (N_0, Y_0, Z_0) = (0, y, z)],
\end{align}
which satisfies the system of PDEs
\begin{align}\label{PDE2}
\left\{
\begin{array}{ccc}
    \partial_t u_+&=&-\partial_y u_+-i\lambda (u_+-u_-)\\
    \partial_t u_-&=&\partial_y u_-+i\lambda(u_+-u_-)
    \end{array}\right.
\end{align}
with $u(0,y,z)=f(y,z)$ and where $u_{\pm}(t,y):=u(t,y,\pm 1)$. 
This system of equations is similar to a linear Boltzmann transport system (see \cite{Kac_1974}),  except that the interaction term is now a phase interaction term; it is termed a system of "quantum transport equations" in \cite{Molfetta_Debbasch_2012}. Our novelty here is to provide a probabilistaic formulation of the solution to the PDE. 
\end{remark}

\begin{remark}
    By rescaling time, $t\to vt$, with a velocity parameter $v>0$, and rescaling the Poisson intensity, $\lambda\to \frac{\lambda}{v}$, the system \eqref{PDE1} becomes
    \begin{align}\label{PDE3}
\left\{
    \begin{array}{ccc}
    \partial_t \psi_+&=&-v\partial_y \psi_++i\lambda \psi_-, \\
    \partial_t \psi_-&=&v\partial_y \psi_-+i\lambda\psi_+. 
    \end{array}
    \right.
\end{align}
\end{remark}

\begin{remark} In contrast to the discrete analogue \eqref{eqns: Feynman Formula} (see the discussion in section \ref{sec: MonteCarloSimulations}), the limit \eqref{eqns: scaling_quantum_expectation} is suitable for the Monte Carlo simulation: for $t \ge 0 $, $y \in \mathbb R$ 
\begin{align}  \label{eq: MonteCalro2}
\psi(t, y, \pm 1) = \psi_{\pm} (t, y) \approx e^{n\theta}\frac{1}{M}\sum_{k=1}^M i^{N^{(k)}_t} f\Big( y \mp \int^{t}_{0} (-1)^{N^{(k)}_{s}} \mathrm d s, \pm (-1)^{N^{(k)}_{t}}\Big), 
\end{align}
where $(N^{(k)}_\cdot, k = 1, \ldots , M)$ are independent Poisson processes with intensity $\lambda$. With $f(y, 1) = f(y,-1) = e^{- \lvert y\rvert/2}/2$, $y \in \mathbb R$, $\lambda = 0.1$, $M=5000$, the estimated probability density function $\lvert \psi_+(t,y)\rvert^2 + \lvert \psi_-(t,y) \rvert^2$ from \eqref{eq: MonteCalro2} is given in Figure \ref{fig: Telegrapher} for $t = 5$ and $t = 10$. 
\end{remark}
Clearly, our probabilistic representation \eqref{eqns: scaling_quantum_expectation} offers a novel way to compute by Monte Carlo technique the solution to the quantum transport PDE \eqref{PDE1}.

\begin{figure}
\begin{tabular}{cc}
\includegraphics[scale=.4]{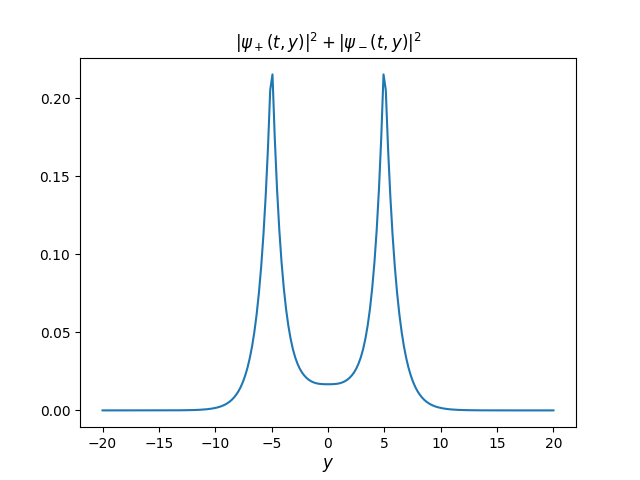}
 & 
\includegraphics[scale=.4]{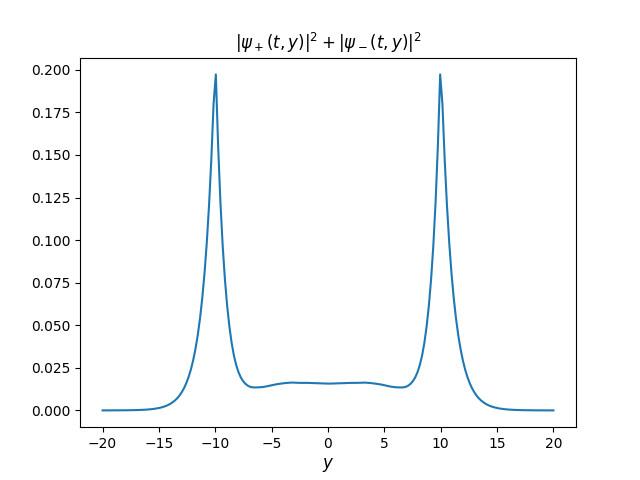} 
\end{tabular}
\caption{Monte Carlo simulation of \eqref{eq: MonteCalro2} for $t = 5$ (left) and $t = 10$ (right).}\label{fig: Telegrapher}
\end{figure}

%% file: Section5.tex
\section{Extensions to other quantum walks}\label{sec: extensions}
In this section, we propose generalization of the Feynman formula (Theorem \ref{thm: Feynman_Formula}) by considering one-dimensional quantum walks with other coin operators, which include more general homogeneous coins, time-dependent coins (see, for example, \cite{Banuls2006}) and site-dependent coins (see, for example, \cite{WojcikEtAl2004}, \cite{MS-2020}, \cite{MS4-2018} and \cite{LindenSharam2009}), and discuss, if possible, their asymptotic behavior and limiting PDEs. We show that all such quantum walks admit a Feynman-type formula, which justifies the robustness of writing a Feynman-type formula. 

In the following $(S_n, Z_n, X_n)_{n\geq 1}$ are always the discrete-time stochastic processes defined in Definition \ref{defn: MA Model} which are driven by the sequence of random variables $(N_j)_{j=1}^\infty$. To write more general Feynman formulae, the sequence $(N_j)_{j=1}$ does not necessarily consist of independent Poisson variables and we shall specify our assumptions on $(N_j)$ for each Feynman formula we derive. Meanwhile we define the filtration $\{\F_n:= \sigma(N_1, \cdots, N_n)\}_{n\geq 1}$ and denote $\E_{s, x, z}$ the expectation conditional on the initial value $(S_0, X_0,Z_0) = (s, x, z)$ for every $(s,x, z)\in \{0, 1, 2, 3\}\times \{\pm 1\}\times \Z $. 
Furthermore, given a sequence of integers $(k_n)_{n\geq 1}$, we define as in Lemma \ref{lemma: Poisson-like_coefficients} the sequences of real numbers $(z_n, x_n)$ by $z_m:= (-1)^{k_m}z_{m-1}$ and $x_m:= x_0-\sum_{j=0}^{m-1} z_j $ for $m\geq 1$ with $(z_0, x_0):=(z, x)$. 
\subsection{General unitary coins}
General unitary matrices can be written in the form 
\begin{align}\label{eqns: general_coin_defn}
    e^{i\phi}Z_\beta e^{i\theta\sigma_1}Z_\alpha &= 
e^{i\phi}\begin{bmatrix}
    \beta & 0\\
    0 & \beta^{-1}
\end{bmatrix}
\begin{bmatrix}
    \cos \theta & i \sin \theta \\
    i \sin \theta & \cos\theta
\end{bmatrix}
\begin{bmatrix}
    \alpha & 0 \\
    0 & \alpha^{-1}
\end{bmatrix}, \\
\quad \mbox{with}\quad  
Z_\omega&:=
\begin{bmatrix}
    \omega & 0\\
    0 & \omega^{-1}
\end{bmatrix}\quad \forall \omega\in \C,
\end{align}
where  $(\alpha, \beta)\in \C^2, \,\abs{\alpha}=\abs{\beta}=1 $ and $ \phi, \theta\in \R$. In this section we are generalizing the Feynman formula in Theorem \ref{thm: Feynman_Formula} to the case of arbitrary general unitary matrices.

In the following, we write $\psi_n^{\phi; \beta, \alpha; \theta}$ to denote the nth step state of the homogeneous quantum walk using the coin $e^{i\phi}Z_\beta e^{i\theta \sigma_1} Z_\alpha$, and $\psi_n^{\theta}:= \psi_n^{0; 1, 1; \theta}$ is the nth step state of the quantum walk with homogeneous coin $e^{i\theta \sigma_1}$. 

\begin{theorem}[Feynman formula for general coins]
For every $x\in \Z, z\in \{\pm 1\}$ and initial state $\psi_0 \in \ell_2(\Z)\otimes \ell_2(\{\pm 1\})$, we have 
\begin{align}\label{eqns: Feynman_Forumla_general_coin}
    \psi_n^{\phi; \beta, \alpha ;\theta}(x, z) = e^{in\phi}{e^{n\theta}} \E_{0,z,x}\left[i^{S_n}(\alpha \beta)^{x-X_n}\alpha^{Z_n-{z}} \psi_0(X_n, Z_n)\right], 
\end{align} 
where the driving random variables are such that $N_j\sim Poi(\theta)$ for every $j\geq 1$ and are independent. 
\end{theorem}
\begin{proof}
Similarly to the proof of Lemma \ref{lemma: Poisson-like_coefficients}, we can derive that 
\begin{equation}\label{eqns: general_coin_expansion}
\begin{split}
   & \psi_n^{\phi; \beta, \alpha; \theta}(x, z) \\
    & =e^{in\phi}\sum_{k_1, \cdots, k_n=0}^\infty \frac{(i\theta)^{k_1+\cdots + k_n}}{k_1 ! \cdots k_n!} \alpha^{(-1)^{k_1}z + \cdots + (-1)^{k_1+\cdots + k_n}z} \\
    & \qquad\qquad\qquad\times \beta ^{z+(-1)^{k_1}z+\cdots + (-1)^{k_1+\cdots + k_{n-1}}}\\
     & \qquad\qquad\qquad\times \psi_0(x - z -(-1)^{k_1}z - \cdots -(-1)^{k_1+\cdots + k_{n-1}}z, z(-1)^{k_1 + \cdots + k_n}z),  
\end{split}    
\end{equation}
    for every $x\in \Z$ and $z\in \{\pm 1\}$. The result follows by taking out a factor or $e^{n\theta}$ as in the proof of Theorem \ref{thm: Feynman_Formula} and from Definition \ref{defn: MA Model}. 
\end{proof}
As in \eqref{eqns: Feynman_Formula_Probability_difference}, if we denote $\psi_n^{\xi; \phi; \beta, \alpha; \theta}$ the nth step state of the respective quantum walk with initial condition $\psi^\xi_0(x, z):= \I_{x=0, z = \xi}$ for every $\xi \in \{\pm 1\}$, then \eqref{eqns: Feynman_Forumla_general_coin} further reduces to
\begin{align}\label{eqns: general_rotation_coin_relation}
    \psi_{n}^{\xi; \phi; \beta, \alpha ;\theta}(x, z)& = e^{in\phi}e^{n\theta} (\alpha\beta)^x \alpha^{\xi-z}\E_{0,x,z}(i^{S_n}\I_{X_n =0, Z_n = 0}) \\
    &= e^{in\phi}(\alpha\beta)^x \alpha^{\xi-z} \psi_n^{\xi; \theta}(x, z). \notag 
\end{align}
This immediately gives us a way to extend the ballistic weak limit of Theorem  \ref{thm: ballistic_limit_for_rotation} from rotation coins $e^{i\theta \sigma_1}$ to general coins:
\begin{corollary}
 With the initial state, 
\begin{align}
    \psi_0(x, z):= R\I_{x=0, z=1} + L\I_{x=0, z= -1}  \quad \mbox{where}\quad (R, L)\in \C^2 \quad \mbox{with}\quad \abs{R}^2 + \abs{L}^2=1,
\end{align}
we have $\frac{1}{n}\Xi_n \Rightarrow Z$ where $Z$ is a real-valued random variable with probability density defined by
\begin{align}
    f_K(a; \theta) \{1+a[\abs{R}^2 - \abs{L}^2+2\tan \theta\Im(\alpha^{2}R\overline{L})a ]\}
\end{align}
\end{corollary}
\begin{proof}
From \eqref{eqns: general_rotation_coin_relation}, we have 
\begin{align}
    \abs{\psi_n^{\xi; \phi; \beta, \alpha; \theta}(x, z)}^2 &= \abs{\psi_n^{\xi; \theta}(x, z)}^2 \quad \mbox{for every}\quad \xi = \pm 1\\
    \psi_n^{+;\phi; \beta, \alpha; \theta}(x, z) \overline{  \psi_n^{-; \phi; \beta, \alpha; \theta}(x, z)} &=\alpha^2 \psi_n^{+; \theta}(x, z) \overline{  \psi_n^{-;  \theta}(x, z)}. 
\end{align}
Using \eqref{eqns: quantum_sum_limit_0}, \eqref{eqns: quantum_sum_limit_1} and \eqref{eqns: quantum_sum_limit_cross} and substituting the respective result in \eqref{eqns: decomposing_cdf} gives the limit  $\lim_n \P(y_1< \frac{\Xi_n}{n} <y_2)$ to be: 
\begin{align}
   \int_{y_1}^{y_2} f_K(a; \theta)\{1+a[\abs{R}^2 - \abs{L}^2 +2\tan\theta \Im (\alpha^2 R\overline{L})]da\},
\end{align}
which completes the proof. 
\end{proof}
\begin{remark}
    Observe that the limiting distribution is independent of the parameters $\beta$ and $\phi$ of the coin. 
\end{remark}

Next, using the stochastic processes $(S_n, T_n)$ with the specific initial conditions $(S_0, T_0) = (0, 0)$ as in \eqref{eqns: Feynman_Formula_Tn}, we can rewrite the Feynman formula for the general coins \eqref{eqns: Feynman_Forumla_general_coin} as:
\begin{align}\label{eqns: Feynman_Forumla_general_coin_T_n}
    \psi_n^{\phi; \beta, \alpha ;\theta}(x, z) = e^{in\phi}{e^{n\theta}} \E\left[i^{S_n}(\alpha \beta)^{zT_n}\alpha^{z(-1)^{S_n}-z} \psi_0(x-zT_n,z(-1)^{S_n})\mid (S_0, T_0) = (0, 0)\right].
\end{align}
Since $\abs{\alpha} = \abs{\beta} = 1$, we can write $\alpha = e^{iA}$ and $\beta := e^{iB}$ for $A, B\in [0, 2\pi)$. By additionally rescaling $A, B$ and $\phi$, we can obtain: 
\begin{theorem}
Let $f:\R\times \{\pm 1\}\to \C$ and $\psi_0^{\epsilon}$ be the same as those in Theorem \ref{thm: quantum_scaling_limit}. Denoting $\psi_n^{\phi; \beta, \alpha; \theta, \epsilon}(x, z)$ the $n$th step state using the coin $e^{i\phi}Z_\beta e^{i\theta \sigma_1}Z_\alpha$ on the initial condition $\psi_0^\epsilon$, we have the point-wise limit:
\begin{align}
\psi(t, y, z,\eta,  B, A, \lambda) &:= \lim_{\epsilon\to 0^+} \psi_{[t/\epsilon]}^{\epsilon \eta; e^{i\epsilon B}, e^{i\epsilon A}; \epsilon \lambda, \epsilon}([y/\epsilon], z) \\
& \, = e^{it\eta +t\lambda} \E\left[i^{N_t} e^{i(A+B)zT_t}f(y-zT_t, z(-1)^{N_t})\right],
\end{align}
for every $t, \eta, B, A, \lambda>0, y\in \R, z = \pm 1$. The limit $\psi_{\pm} (t, y) := \psi(t, y, \pm 1, \eta,B, A, \lambda)$ satisfies the system of PDEs: 

\begin{equation}
\left \{ \begin{array}{ccl} 
\partial_t \psi_+ & = &  -\p_y \psi_+ + i\lambda \psi_- + i(A+B)\psi_+ +i\eta\psi_+, \\
\partial_t \psi_- & = & \p_y \psi_- + i\lambda \psi_+ - i(A+B)\psi_-+i\eta\psi_-  
\end{array}
\right . 
\end{equation}
with $\psi_\pm(0, y) = f(y, \pm 1)$. 
\end{theorem}
\begin{proof}
Similarly to the proof of Theorem \ref{thm: quantum_scaling_limit}, one can show that 
\begin{align}\label{eqns: general_coin_pde}
    &\psi(t+h, y, z, \eta, B, A , \lambda) \\
    =& e^{h\lambda + ih \eta} \E\left[i^{N_h} e^{i(A+B)zT_h} \psi(t, y-zT_h, z(-1)^{N_h}, \eta, B, A, \lambda)\right], 
\end{align}
and the remainder of the proof follows the lines of the proof of Theorem \ref{thm: quantum_scaling_limit}.  
\end{proof}

\begin{remark}
    Note that \eqref{eqns: general_coin_pde} differs from \eqref{PDE1} only by phases driven by $A, B$ and $\eta$ in the general coin $\eqref{eqns: general_coin_defn}$, while the interaction term remains the same driven by the component $e^{i\theta \sigma_1}$ of the coin. 
\end{remark}

\subsection{Site-dependent coins}
A quantum walk is space-inhomogeneous when the coin operator is site-dependent, which can be written as a controlled operator:
\begin{align}
    C = \sum_{x\in \Z} e_xe_x^* \otimes  C_x, \quad C_x\in U(\{\pm 1\}) \quad \mbox{for all}\quad x\in \Z. 
\end{align}
In the following we consider two types of site-dependent coins: a rotation coin with site-dependent global phase and a rotation coin with site-dependent angles of rotation. 
\subsubsection{Rotation coin with site-dependent global phase} 
The coin under consideration in this section is 
\begin{align}
    C = \sum_{x\in \Z} e_x e_x^* \otimes e^{iV(x)}e^{i\theta \sigma_1}, 
\end{align}
where $V:\Z\to \R$ is some real-valued function and $\theta >0$. Denoting by $\psi_n^{V; \theta}$ the $n$th step state of the quantum walk with this coin, we have: 
\begin{theorem}
For every $x\in \Z, z\in \{\pm 1\}$ and initial state $\psi_0 \in \ell_2(\Z)\otimes \ell_2(\{\pm 1\})$:
\begin{align}
    \psi_n^{V; \theta}(x, z) = e^{n\theta} \E_{0,x,z}\left[ i^{S_n} e^{i[V(X_1)+\cdots +V(X_n)]} \psi_0(X_n, Z_n)\right], 
\end{align}
where the driving random variables are independent with $N_j\sim Poi(\theta)$ for every $j\geq 1$. 
\end{theorem}
\begin{proof}
Similar to the proof of Lemma \ref{lemma: Poisson-like_coefficients}, we can derive
\begin{align}
    \psi_n^{V; \theta}(x, z) = \sum_{k_1, \cdots, k_n=0}^\infty& \frac{(i\theta)^{k_1+\cdots + k_n}}{k_1 ! \cdots k_n!}\\
    & \times e^{i[V(x-z) +V(x-z-(-1)^{k_1}z)-\cdots -V(x-z-(-1)^{k_1}z-\cdots -(-1)^{k_{n-1}}z)]}\\
    \times & \psi_0(x-z-(-1)^{k_1}z-\cdots -(-1)^{k_1+\cdots +k_{n-1}}, z(-1)^{k_1+\cdots + k_n}). \notag 
\end{align}
for every $n\geq 1$. The Feynman formula easily follows from factoring out $e^{n\theta}$. 
\end{proof}

\begin{corollary}
    Let $V:\R \to \R$ be a real-valued continuous function. Let $f:\R\times \{\pm 1\}\to \C$ and $\psi_0^\epsilon$ be the same as those in Theorem \ref{thm: quantum_scaling_limit}. Denoting  $\psi_n^{V, \theta; \epsilon}(x, z)$ the nth step state using the initial state $\psi_0^\epsilon$, we have the point-wise limit:
\begin{align}
    \psi(t, y, z) & := \lim_{\epsilon\to 0} \psi_{[t/\epsilon]}^{\epsilon V(\cdot\epsilon ), \epsilon \lambda; \epsilon}([y/\epsilon], z) \\ 
    & = e^{t\lambda} \E\left[i^{N_t} e^{i\int_0^t V(y-zT_s) \mathrm d s  }f(y - zT_t, z(-1)^{N_t})\right], 
\end{align}
where $N_t$ is a Poisson process with parameter $\lambda$ and $T_t = \int_0^t (-1)^{N_s}\mathrm ds$, for every $t>0, y\in \R, z= \pm1, \lambda>0 $.
\end{corollary}
\begin{proof}
    It is similar to the proof of Theorem \ref{thm: quantum_scaling_limit}. 
\end{proof}
\begin{remark}
    We could derive a PDE similar to \eqref{PDE1} but with additional space-dependent phase terms. 
\end{remark}
\subsubsection{Rotation coin with site-dependent angles} 
In the following we consider the coin
\begin{align}
    C = \sum_{x\in \Z} e_x e_x^* \otimes e^{i\Theta(x)\sigma_1}, 
\end{align}
where $\Theta:\Z\to \R$ is some real-valued function. We denote by $\psi_n^{\Theta(\cdot)}$ the $n$th step state of the quantum walk with this coin. Unlike all other coins the Feynman formula for this coin is not driven by independent Poisson random variables:
\begin{theorem}\label{thm: site_dependent_angle_coin}
For every $x\in \Z, z = \pm 1$ and initial state $\psi_0\in \ell_2(\Z)\otimes \ell_2(\{\pm 1\})$, we have
\begin{align}
    \psi_n^{\Theta(\cdot)}(x, z) = \E_{0,x,z}\left[i^{S_n} e^{\Theta(X_1)+\cdots + \Theta(X_n)} \psi_0(X_n, Z_n)\right], 
\end{align}where  the driving random variables $(N_j)$ are defined by 
\begin{align}
    N_j\mid \F_{j-1}\sim Poi(\Theta(X_j)), \quad \F_{j}:= \sigma(N_1, \cdots, N_j)\quad \mbox{for every}\quad j\geq 1. 
\end{align} 
\end{theorem}
\begin{proof}
We provide a proof by induction: first, it is straight forward to see that 
\begin{align}\label{eqns: general_coin_base_case}
    (W \psi)(x, z) &= \sum_{k=0}^\infty \frac{i^k}{k!} (\Theta(x-z))^k \psi(x-z, (-1)^kz) \\
    &= e^{\Theta(x-z)}\E\left[i^{N} \psi(x-z, (-1)^Nz)\mid N\sim Poi(\Theta(x-z))\right], \notag 
\end{align}
where $W$ is the walk operator and $\psi\in \ell^2(\Z)\otimes \ell^2(\{\pm 1\})$ is any state. Then for every $n\geq 2$, we have:
\begin{align}
& \E\left[i^{S_n}e^{\Theta(X_1)+\cdots + \Theta(X_n)}\psi_0(X_n, Z_n)\mid \F_{n-1}\right]\\ \notag 
 =&i^{S_{n-1}}e^{\Theta(X_1)+\cdots + \Theta(X_n)} \E\left[ i^{N_n} \psi_0(X_n, Z_{n-1}(-1)^{N_n})\mid N_n\sim Poi(\Theta(X_n))\right] \\ \notag 
 =& i^{S_{n-1}}e^{\Theta(X_1)+\cdots + \Theta(X_n)} e^{-\Theta(X_n)} (W\psi_0)(X_n+Z_{n-1}, Z_{n-1})\\
 =& i^{S_{n-1}}e^{\Theta(X_1)+\cdots +\Theta(X_{n-1})} (W\psi_0)(X_{n-1}, Z_{n-1}), \notag 
\end{align}
where the second equality follows from \eqref{eqns: general_coin_base_case}. Then we have from the tower property: 
\begin{align*}
   & \E_{0, x, z}\left[i^{S_n}e^{\Theta(X_1)+\cdots + \Theta(X_n)}\psi_0(X_n, Z_n)\right] \\
   &= \E_{0, x, z}\left[\E\left[i^{S_n}e^{\Theta(X_1)+\cdots +\Theta(X_n)}\psi_0(X_n, Z_n)\mid \F_{n-1}\right]\right] \\ \notag 
    &= \E_{0, x, z}(i^{S_{n-1}}e^{\Theta(X_1)+\cdots +\Theta(X_{n-1})} (W\psi_0)(X_{n-1}, Z_{n-1}))\\
    &= (W^{n-1} W\psi_0)(x, z) = (W^n\psi_0)(x, z) \notag, 
\end{align*}
where the second last equality follows from the induction hypothesis: for every $k\leq n$, we have
\begin{align}
    (W^k\psi)(x, z) = \E_{0, x, z}(i^{S_k}e^{\Theta(X_1)+\cdots + \Theta(X_k)}\psi(X_k, Z_k))
\end{align}
for every $x\in \Z$  and $z = \pm 1$ and every state $\psi\in \ell^2(\Z)\otimes \ell^2(\{\pm 1\})$.

\end{proof}

\subsection{Time-dependent coins}
A quantum walk is time-inhomogeneous when the coin operator is time-dependent; its dynamics is given by 
\begin{align}
    \psi_{n} = W_{n} \psi_{n-1}\quad \mbox{for}\quad n\geq 1; \quad \psi_0 \in \ell_2(\Z)\ox \ell_2(\pm 1), 
\end{align}
in which the walk operator is defined by $W_n:= SC_n$ where $(C_n)_{n=1}^\infty$ is a sequence of coin operators. 

In the following we consider the cases both with and without an additional site dependence. 
\subsubsection{Time-dependent, site-homogeneous coins}
The most general time-dependent, site-homogeneous  is given by the sequence
\begin{align}
    C_n = I\otimes e^{i\phi_n}Z_{\beta_n} e^{i\theta_n \sigma_1}Z_{\alpha_n},
\end{align} 
where $(\alpha_n, \beta_n)\in \C^2$ with $\abs{\alpha_n} = \abs{\beta_n}=1$ and $\theta_n, \phi_n\in (0, 2\pi)$ is a real sequence. Denote the $n$th step state $\psi_n^{(\theta_j); (\beta_j), (\alpha_j), (\phi_j)}$ of the quantum walk with these coins. Its Feynman formula is given by: 
\begin{theorem}
For every $x\in \Z, z\in \{\pm 1\}$ and initial state $\psi_0\in \ell_2(\Z)\otimes \ell_2(\{\pm 1\})$, 
\begin{align}
    \psi_n^{(\theta_j); (\beta_j), (\alpha_j), (\phi_j)}(x, z) = e^{i\sum_{j=1}^n \phi_j +\sum_{j=1}^n \theta_j} \E_{0, x, z}\Big[i^{S_n} \prod_{j=1}^n \alpha_j^{Z_j} \prod_{j=1}^n \beta_j^{Z_{j-1}}\psi_0(X_n, Z_n)\Big] , 
\end{align}
where the driving process is given by independent $N_j\sim Poi(\theta_j)$  for every $j\geq 1$. 
\end{theorem}
\begin{proof}
    The proof follows from a slight modification to \eqref{eqns: general_coin_expansion}. 
\end{proof}
A scaling limit can be derived when we restrict to the case that only $(\theta_j)$ is time-varying:

\begin{corollary}
   Let $\lambda:[0, \infty)\to (0, \infty)$ be a positive, continuous function. Let $f:\R\times \{\pm 1\}\to \C$ and $\psi_0^\epsilon$ be the same as those in Theorem \ref{thm: quantum_scaling_limit}. Denoting  $\psi_n^{(\theta_j); \epsilon}(x, z)$ the nth step state using the time-dependent coins $C_n:= I\otimes e^{i\theta_n \sigma_1}$ and the initial state $\psi_0^\epsilon$, we have the point-wise limit
\begin{align}
    \psi(t, y, z, \lambda):= \lim_{\epsilon\to 0^+} \psi_{[t/\epsilon]}^{(\epsilon \lambda(\epsilon j))_{j\in \Z}}([y/z], z) = e^{\int_0^t\lambda(s)\mathrm ds} \E\left[i^{N_t} f(y-zT_t, z(-1)^{N_t})\right],
\end{align}
where $(N_t)_{t\geq 0}$ is a Poisson process with intensity level $\lambda (t)$ and $T_t:= \int_0^t (-1)^{N_s}\mathrm ds$. 
\end{corollary}
\begin{proof}
    Similar to the proof of Theorem \ref{thm: quantum_scaling_limit}. 
\end{proof}

\subsubsection{Time-dependent, site-inhomogeneous coins}
We consider the time-dependent coins
\begin{align}
    C_n:= \sum_{x\in \Z} e_xe_x^* \otimes e^{i\Theta_n(x) \sigma_1} \quad \mbox{for all}\quad n\geq 1,
\end{align}
where $\Theta_n:\Z\to \R$ is a real-valued function for $n \ge 1$. Denoting the $n$th step state $\psi_n^{(\Theta_j(\cdot))}$ of the quantum walk with this coin, we have
\begin{theorem}
For every $x\in \Z, z = \pm 1$ and initial state $\psi_0\in \ell_2(\Z)\otimes \ell_2(\{\pm 1\})$, we have
\begin{align}
    \psi_n^{(\Theta_j(\cdot))}(x, z) = \E_{0, x, z}[i^{S_n} e^{\Theta_1(X_1)+\cdots + \Theta_n(X_n)} \psi_0(X_n, Z_n)], 
\end{align}
where  the driving random variables $(N_j)$ are defined by 
\begin{align}
    N_j\mid \F_{j-1}\sim Poi(\Theta_j(X_j)), \quad \F_{j}:= \sigma(N_1, \cdots, N_j)\quad \mbox{for every}\quad j\geq 1. 
\end{align} 
\end{theorem}
\begin{proof}
    It follows from a slight modification to the proof of Theorem \ref{thm: site_dependent_angle_coin}. 
\end{proof}

\subsection{Yet another walk}
Let $M$ be the transition matrix of a classical symmetric nearest-neighbor random walk on $\Z$: that is $M(x,y)=\frac{1}{2}$ if $y=x\pm1$ and $0$ if not for every $x, y\in \Z$. Note that the following discussion generalizes easily to the multi-dimensional case, but for simplicity we stay with the one-dimensional case.  Indeed, $M$ is not  unitary but $e^{i\lambda M}$ is and we consider the amplitude dynamics:
\[
\psi_{n+1}(x)=e^{i(\lambda M+V(x))}\psi_n(x),
\]
where $V$ is a bounded potential and $\psi_0$ is an initial quantum state, $\psi_0=\delta_0$ for instance. 
\begin{theorem}\label{MolchanovDiscrete}
For a sequence of i.i.d. Poisson random variables $(N_1,N_2,\cdots)$ with parameter $\lambda>0$, we denote $S_n=N_1+\cdots+N_n$ and, given a classical random walk $(Y_n)$ with transition matrix $M$, we define $Y_0=X_0=x$ and $X_n=Y_{S_n}$ for all $n\geq 1$.
We have the following Feynman formula:
 \begin{align}  
 \psi_n(x)&=e^{\lambda n}\E_{0,x}\left[i^{S_n}e^{i(V(X_1)+\cdots+V(X_n))}\psi_0(X_n)\right].
 \end{align}
\end{theorem}
\begin{proof}
    By induction, conditioning on $N_1$, and using time-homogeneity of the chain $(X_n)$, one has:
    \begin{align*}
       \psi_{n+1}(x)&=e^{\lambda (n+1)}\E_{0,x}\left[i^{S_{n+1}}e^{i(V(X_1)+\cdots+V(X_{n+1}))}\psi_0(X_{n+1})\right] \\
       &=e^{\lambda (n+1)}\sum_{k=0}^\infty\frac{ e^{-\lambda}\lambda^k}{k!}i^k
\E_{0,x}\left[i^{N_2+\cdots+N_{n+1}}e^{i(V(Y_k)+V(X_2)\cdots+V(X_{n+1}))}\psi_0(X_{n+1})\right] \\       &=e^{\lambda (n+1)}\sum_{k=0}^\infty\frac{ e^{-\lambda}\lambda^k}{k!}i^k M^ke^{iV(x)}\E_{0,x}\left[i^{S_{n}}e^{i(V(X_1)+\cdots+V(X_{n}))}\psi_0(X_{n})\right] \\
&=e^{\lambda n}e^{i\lambda M}e^{iV(x)}\E_{0,x}\left[i^{S_{n}}e^{i(V(X_1)+\cdots+V(X_{n}))}\psi_0(X_{n})\right]\\
&=e^{i(\lambda M+V(x))}\psi_n(x).
    \end{align*}
\end{proof}
The previous result has a continuous-time version which was known for some time as  Molchanov formula (see \cite{Carmona-StFlour}). Let $(X_t)_{t\geq 0}$ be the continuous-time Markov chain with transition matrix $M$ and $(N_t)_{t\geq 0}$ be the Poisson process with intensity $\lambda>0$ describing the jump times of the chain. Given a bounded potential $V(x)_{x\in \Z}$, we have the following:
\begin{proposition}\label{Molchanov}
   Consider the continuous-time discrete-space Schr\"odinger equation
    \begin{align}
        i\frac{\partial\psi}{\partial t}= -(\lambda M+V(x))\psi.
    \end{align}
     For a given function $\psi(0,x)$, its solution $\psi(t,x)=e^{it(\lambda M+V(x))}\psi(0,x)$
 admits the probabilistic representation:
 \begin{align}
 \psi(t,x)
 &=e^{\lambda t}\E_{0,x}\left[i^{N_t}e^{i\int_0^tV(X_s)\mathrm ds}\psi(0,X_t)\right].
 \label{eq:Molchanov}
 \end{align}
\end{proposition}
\begin{proof}
    Starting from  formula \eqref{eq:Molchanov} at $t+h$, conditioning on $N_h$, and using time-homogeneity, we have:
    \begin{align*}
        \psi(t+h,x)&=e^{\lambda(t+h)}\E_{0,x}\left[i^{N_{t+h}}e^{i\int_0^{t+h}V(X_s)ds}\psi(0,X_{t+h})\right]\\
        &=e^{\lambda(t+h)}e^{-\lambda h}e^{ihV(x)}
\E_{0,x}\left[i^{N_{t}}e^{i\int_0^tV(X_s)ds}\psi(0,X_{t})\right]\\
&\quad +e^{\lambda(t+h)}e^{-\lambda h}(\lambda h)iM\E_{0,x}\left[i^{N_{t}}e^{i\int_0^tV(X_s)ds}\psi(0,X_{t})\right]+O(h^2),
    \end{align*}
    where we only considered $N_h=0$ and $N_h=1$  since $\Pr[N_h\geq 2]=O(h^2)$. Therefore,
    \begin{align}
      \psi(t+h,x)-\psi(t,x)=ih\left( (V(x)+ \lambda M)\psi(t,x)\right) +O(h^2),
    \end{align}
    and dividing by $h$ and taking the limit $h\to 0$ concludes the proof.
\end{proof}
\newpage

%% file: Conclusion.tex
\section{Conclusion}
\subsection{Summary}
We introduced a Feynman-type (discrete path-integral) representation that links one-dimensional quantum walks on $\mathbb{Z}$ to a four-state Markov
additive process. This representation provides a unified probabilistic handle on the
dynamics: it isolates a tilted-kernel structure whose spectral decomposition explains
the ballistic regime, and it yields a transparent route to continuum limits. Using this
framework, we demonstrated the derivation of the well-known ballistic weak limit for the position distribution, and we
derived a probabilistic representation to the quantum transport PDE, which is a system of transport PDEs with phase interaction term and is the hyperbolic scaling
limit of quantum walks. Meanwhile, the Feynman-type representation is robust across different variants of one-dimensional quantum walks, which include both time-inhomogeneous and site-inhomogeneous coins. Also we noted that our probabilistic representation offers a Monte Carlo way to compute the solution for this type of PDE. 

A distinctive feature of our approach is that both asymptotic descriptions --- ballistic
limits of quantum walks and large deviation tails of Markov additive processes---arise from the same tilted kernel machinery. The spectral branch that is responsible for the large deviation properties is removed when applying the Feynman formula to quantum walks; the asymptotic analysis then follows from a typical stationary-phase analysis of the remaining spectral branch. On the continuum side, our Feynman formula naturally provides a probabilistic formulation of the solution to the quantum PDE using also  hyperbolic scaling limit of the Markov additive process. 

\subsection{Outlook}

In Section \ref{sec: extensions}, we demonstrated the robustness of the Feynman formula on inhomogeneous coins; yet these coins can yield non-ballistic behaviors, including localized and hybrid behaviors, as seen in the numerical experiments by \cite{AhmadSajjadSajid2020} for site-dependent coins and \cite{Panahiyan2018} for time-dependent coins. It is natural to ask how these properties are reflected in the associated Markov additive process. Meanwhile, quantum walks with higher-dimensional state space and chirality space allow one to study entanglement effects involved in a quantum walk (see \cite{WKKK_2008}  which discussed quantum walks on $\Z^2$ and \cite{VenegasAndraca2005} which discussed one-dimensional quantum walks with higher dimensional, entangled coins). 
It is natural to extend the robustness of the Feynman formula to these higher-dimensional generalizations. Lastly, we are also currently working on probabilistic representations for the nonlinear quantum walks and their associated Dirac PDEs described in \cite{MS4-2018} and \cite{MS-2020}.

%% file: Appendix1.tex
\section*{Sketch of proof of Lemma \texorpdfstring{\ref{lemma: integral_asymptotics}}{}}\label{appendix: fourier} \setcounter{section}{1}
The method of stationary phase (see for instance \cite{Olver_1997} and \cite{Wong_saddle_2001} for details) entails that Fourier integrals of the form
\begin{align}\label{eqns: Fourier_integrals}
    \int_a^b g(k) e^{in\Theta(k)}\mathrm dk, 
\end{align}
where $\Theta$ is real-valued and $g$ is either real or complex-valued,
admit the asymptotic relation\footnote{We say that two sequences $f, g:\N\to \R$ are asymptotically equivalent, written $f(n)\sim g(n)$ as $n\to \infty$, if $f(n) = g(n)(1+o(n))$ as $n\to\infty$; equivalently $\lim_{n\to \infty}\frac{f(n)}{g(n)}=1$.} 
\begin{align}\label{eqns: stationary_phase_heuristic}
      \int_a^b g(k) e^{in\Theta(k)}\mathrm dk \sim \sum_{j=1}^N g(k_j)e^{in\Theta(k_j)}\sqrt{\frac{2\pi i }{n\Theta''(k_j)}}\quad \mbox{as}\quad n\to \infty,
\end{align}
where $\{k_j\}_{j=1}^N$ are critical points to $\Theta(k)$ that are interior to $[a, b]$ for some $N\in \mathbb{N}$ under certain regularity assumption on $g$ and $\Theta$. In this section we would use the method of stationary phase to both derive \eqref{eqns: integral_quantum_limit_0}, \eqref{eqns: integral_quantum_limit_1} and \eqref{eqns: integral_quantum_limit_cross} formally and to justify all assumptions we made in the formal computation in \ref{sec: rigorous_justificaiton}.
\subsection{The limit \texorpdfstring{\eqref{eqns: integral_quantum_limit_0}}{in Lemma 3.6}}
\subsubsection{Formal computation of single application of stationary phase}
We start with  by formally applying the method of stationary phase for a fix $a\in (-\cos \theta + \epsilon, \cos\theta -\epsilon )$ on \eqref{eqns: integral_quantum_limit_0},  assuming that we can replace $\ceil{an}$ with $an$, when $n$ is large enough (see \ref{sec: rigorous_justificaiton} for a justification):  
\begin{align}\label{eqns: replace_with_an}
    e^{n\theta} \sum_{j=1, 2}I_n(0, an, B_j)
   & =e^{n\theta}\Big(\frac{1}{2\pi} \int_{-\pi}^\pi r_{B_1}(k, 0)\ell_{B_1}(k, 0) e^{-n\theta +in\cos^{-1}(\cos k \cos \theta )-iank}\mathrm dk \\ 
    & \qquad {}+ \frac{1}{2\pi}\int_{-\pi}^\pi r_{B_2}(k, 0) \ell_{B_2}(k, 0) e^{-n\theta -in\cos^{-1} (\cos k \cos\theta)-iank} \mathrm dk \Big)\notag \\ \label{eqns: Fourier_integrals_no_theta}
    &=\frac{1}{2\pi} \int_{-\pi}^\pi r_{B_1}(k, 0)\ell_{B_1}(k, 0) e^{ in[\cos^{-1}(\cos k \cos \theta )-ak]}\mathrm dk \\
    & \qquad {}+ \frac{1}{2\pi}\int_{-\pi}^\pi r_{B_2}(k, 0) \ell_{B_2}(k, 0) e^{-in[\cos^{-1} (\cos k \cos\theta)-ak]}\mathrm dk \notag, 
\end{align}
which is the sum of two Fourier integrals of the form \ref{eqns: Fourier_integrals}. 

Defining 
\begin{align*}
    \Theta_{B_j}(k; y):=
\begin{cases}
    \cos^{-1}(\cos k \cos \theta ) - ky,\quad j = 1\\
  -   \cos^{-1}(\cos k \cos \theta ) - ky,\quad j = 2
\end{cases}
\end{align*}
and applying the method of stationary phase, that is the asymptotic relation \eqref{eqns: stationary_phase_heuristic},  on \eqref{eqns: Fourier_integrals_no_theta} gives the asymptotic relation: 

\begin{align} \notag 
  e^{n\theta} \sum_{j=1, 2}I_n(0, an, B_j) &\sim \frac{1}{2\pi} r_{B_1} (k_1, 0) \ell_{B_1}(k_1, 0) e^{in\Theta_{B_1}(k_1; a)} \sqrt{\frac{2\pi i}{\Theta_{B_1}''(k_1, a)}}\\ \label{eqns: four_term_decomposition}
  & \qquad {} +\frac{1}{2\pi} r_{B_1} (k_2, 0) \ell_{B_1}(k_2, 0) e^{in\Theta_{B_1}(k_2; a)} \sqrt{\frac{2\pi i}{\Theta_{B_1}''(k_2, a)}} \\ \notag 
  & \qquad {}+\frac{1}{2\pi} r_{B_2} (k_3, 0) \ell_{B_2}(k_3, 0) e^{in\Theta_{B_2}(k_3; a)} \sqrt{\frac{2\pi i}{\Theta_{B_2}''(k_3, a)}} \\
  &\qquad {}+\frac{1}{2\pi} r_{B_2} (k_4, 0) \ell_{B_2}(k_4, 0) e^{in\Theta_{B_2}(k_4; a)} \sqrt{\frac{2\pi i}{\Theta_{B_2}''(k_4, a)}},\notag 
\end{align}
where $k_1 = k_1(a), k_2= k_2(a)$ are critical points of $\Theta_{B_1}$ with respect to $k$, and $k_3=k_3(a), k_4=k_4(a)$ are critical points of $\Theta_{B_2}$ with respect to $k$. In fact, the four critical points satisfy the  relations 
\begin{align}\label{eqns: critical_point_relation}
    k_2 = \pi - k_1, k_3 = -k_2, k_4 = -k_1.
\end{align} 
Using the critical point equation $\Theta'(k_1, a) = 0,$ we obtain 
\begin{align}\label{eqns: critical_point_equation}
    \frac{\sin k_1 \cos \theta}{\sqrt{1-\cos^2 k_1 \cos^2 \theta}} = a,\quad \mbox{and thus} \quad \sin^2 k_1=\frac{a^2}{1-a^2}\tan ^2\theta. 
\end{align}
We specify $k_1$ to be the unique element in $(-\pi ,\pi)$ such that 
\begin{align}\label{eqns: critical_point_definition}
    \sin k_1 = \frac{a}{\sqrt{1-a^2}}\tan \theta, && \cot k_1 = \sqrt{\csc^2 k_1 - 1} = \frac{\sqrt{\cos^2 \theta - a^2}}{\abs{a}\sin \theta }.
\end{align}
Meanwhile, the second derivatives in \ref{eqns: four_term_decomposition} are taken with respect to the parameter $k$ and are given by
\begin{align}\label{eqns: 2nd_derivative_0}
    \Theta''_1(k, a) = \frac{\cos k \cos \theta \sin^2 \theta}{(1-\cos^2 k \cos^2 \theta)^{3/2}}, && \Theta''_2(k, a) = -\frac{\cos k \cos \theta \sin^2 \theta}{(1-\cos^2 k \cos^2 \theta)^{3/2}}, 
\end{align}
for every $k\in \R$. Their values at the saddle points could be computed by noting the identity: 
\begin{align}\label{eqns: 2nd_derivative_at_saddle_0}
    \frac{\cos k_1 \cos \theta \sin^2\theta}{(1-\cos^2 k \cos^2 \theta)^{3/2}} = \sgn(a) \frac{(1-a^2) \sqrt{\cos^2 \theta - a^2}}{\sin \theta } = \frac{\sgn(a)}{\pi f_K(a; \theta)}, 
\end{align}
where $f_K$ is defined in \eqref{eqns: Konnos_distribution}. In addition, the eigenvector terms in \ref{eqns: four_term_decomposition} are given by: 
\begin{equation}\label{eqns: eigenvectors_0}
   \begin{split}
       r_{B_1}(k, 0)\ell_{B_1}(k, 0) &= \frac{1}{4}\Big(1+ \frac{\sin k \cos \theta }{\sqrt{1-\cos^2\theta \cos^2 k}}\Big), \\ r_{B_2}(k, 0)\ell_{B_2}(k, 0) &= \frac{1}{4}\Big(1- \frac{\sin k \cos \theta }{\sqrt{1-\cos^2\theta \cos^2 k}}\Big),  
   \end{split} 
\end{equation}
which can be further simplified using \eqref{eqns: critical_point_equation}. Substituting \eqref{eqns: 2nd_derivative_0} and \eqref{eqns: eigenvectors_0} into \eqref{eqns: four_term_decomposition} with the corresponding values at saddle points given by \eqref{eqns: 2nd_derivative_at_saddle_0} and \eqref{eqns: critical_point_equation}  shows that $e^{n\theta}\sum_{j=1, 2}I_n(0, an, B_j)$ is asymptotically equivalent to:
\begin{align}
    \label{eqns: first_simplification}
    & \frac{1}{2\pi}\frac{1+a}{4}\sqrt{\frac{2\pi^2}{n} f_K(a; \theta)}\\
\times & \{e^{in \cos^{-1}(\cos k_1 \cos \theta)-ik_1an + i\sgn(a)\frac{\pi}{4}} + e^{in \cos^{-1}(\cos k_2 \cos \theta)-ik_2an - i\sgn(a)\frac{\pi}{4}} \\
&+e^{-in \cos^{-1}(\cos k_3 \cos \theta)-ik_3an + i\sgn(a)\frac{\pi}{4}} + e^{-in \cos^{-1}(\cos k_4 \cos \theta)-ik_4an - i\sgn(a)\frac{\pi}{4}}\}.  
\end{align}
Using the fact that $k_4 = -k_1$ and $k_3 = -k_2$ as in \ref{eqns: critical_point_relation}, we can simplify the sum of four exponential terms in \eqref{eqns: first_simplification} into
\begin{align}
\label{eqns: oscillation_terms_simplification}
& \qquad 2\cos (n\Theta_{B_1}(k_1, a) + \sgn (a)\frac{\pi}{4}) + 2\cos(n\Theta_{B_1}(k_2, a) -\sgn(a)\frac{\pi}{4})\\ 
=& 4 \cos (\frac{n}{2}[\Theta_{B_1}(k_1, a) + \Theta_{B_1}(k_2, a)]) \cos (\frac{n}{2}[\Theta_{B_1}(k_1, a) -\Theta_{B_2}(k_2, a)] + \sgn(a)\frac{\pi}{4})\\ \notag 
=&4 \cos (\frac{n\pi (1-a)}{2})\cos(n \Theta_{B_1}(k_1, a)-\frac{n\pi (1-a)}{2} +\sgn(a)\frac{\pi}{4})\\ \notag 
 =&\frac{4}{\sqrt{2}}\cos(\frac{n\pi (1-a)}{2})\Big[\cos(n\Theta_{B_1}(k_1, a) - \frac{n\pi(1-a)}{2}) \\
 & \qquad \qquad \qquad \qquad  \qquad \qquad {} - \sgn(a)\sin(n\Theta_{B_1}(k_1, a) - \frac{n\pi(1-a)}{2})\Big], \notag
\end{align}
where we used the sum-to-product formula in the first equality, and the following identity for the second equality:
\begin{align}\label{eqns: angle_identity}
    \Theta_{B_1}(k_1, a) + \Theta_{B_1}(k_2, a) =n\pi(1-a),
\end{align}
which follows from the relation $k_2 = \pi - k_1$ (see again \ref{eqns: critical_point_relation}) and the identity $\cos^{-1}(-z) = \pi -  \cos ^{-1}z$ for every $z\in (-1, 1)$. Putting \eqref{eqns: oscillation_terms_simplification} into \eqref{eqns: first_simplification}, we have:
\begin{equation}\label{eqns: full_single_stationary_phase}
\begin{split}
    &\frac{1+a}{2\sqrt{n}}\sqrt{f_K(a; \theta)}\cos(\frac{n\pi (1-a)}{2})\\ 
\times & 
\Big[\cos(n\Theta_{B_1}(k_1, a) - \frac{n\pi(1-a)}{2}) - \sgn(a)\sin(n\Theta_{B_1}(k_1, a) - \frac{n\pi(1-a)}{2})\Big]. 
\end{split}
\end{equation}

\subsubsection{Formal computation for integrating leading terms}
In the previous section, we formally showed that $e^{n\theta} \sum_{j=1,2}I_n(0, an, B_j)$ is asymptotically equivalent to \eqref{eqns: full_single_stationary_phase}; in this section we continue from there to formally verify \eqref{eqns: integral_quantum_limit_0}. To begin with, taking modulus square on \eqref{eqns: full_single_stationary_phase} gives the asymptotic relation:
\begin{align}
    &e^{2n\theta} \abs{\sum_{j=1,2}I_n(0, an, B_j)}^2\\
    &\sim \frac{(1+a)^2}{4n }f_K(a; \theta)\cos^2(\frac{n\pi (1-a)}{2})[1-\sgn(a)\sin(2n\Theta_{B_1}(k,a) - n\pi (1-a))] \notag\\
    &= \frac{(1+a)^2}{8n }f_K(a; \theta)(1+\cos(n\pi(1-a))[1-\sgn(a)\sin(2n\Theta_{B_1}(k,a) - n\pi (1-a))]. \notag
\end{align}
It follows that we have the asymptotic relation:
\begin{align} \label{eqns: case_0_integral_asymptotic_computation}
    & \qquad n\int_{\frac{1}{n}\floor{y_1 n}}^{\frac{1}{n}\floor{y_2 n}} e^{2n\theta} \abs{\sum_{j=1,2}I_n(0, an, B_j)}^2 \mathrm d a\\ \notag
\sim& \int_{\frac{1}{n}\floor{y_1 n}}^{\frac{1}{n}\floor{y_2 n}} \frac{1}{8} f_K(a; \theta)(1+a)^2da+\int_{\frac{1}{n}\floor{y_1 n}}^{\frac{1}{n}\floor{y_2 n}}\frac{1}{8} f_K(a; \theta)(1+a)^2\cos(n\pi(1-a)) \mathrm d a\\ \notag
&-\frac{1}{8}\int_{\frac{1}{n}\floor{y_1 n}}^{\frac{1}{n}\floor{y_2 n}}f_K(a; \theta)(1+a)^2\sgn(a) \sin (2n\Theta_{B_1}(k_1, a)-n\pi(1-a))\mathrm d a\\ \notag
&-\frac{1}{8}\int_{\frac{1}{n}\floor{y_1 n}}^{\frac{1}{n}\floor{y_2 n}}f_K(a; \theta)(1+a)^2\sgn(a) \sin (2n\Theta_{B_1}(k_1, a)-n\pi(1-a))\cos(n\pi(1-a)) \mathrm d a.
\end{align}
Note that this asymptotic relation is valid because the domain of integral is  contained in $(-\cos\theta+\epsilon, \cos\theta-\epsilon)$ which is away from the problematic boundary points $\pm \cos\theta$. Due to the existence of fast oscillation terms, as $n\to \infty$, the last three integrals in the asymptotic equivalence converge to $0$ by both the Riemann-Lebesgue Lemma and the method of stationary phase (see \ref{sec: rigorous_justificaiton} for details). The first term converges and we have
\begin{align}
    \lim_{n\to \infty} n\int_{\frac{1}{n}\floor{y_1 n}}^{\frac{1}{n}\floor{y_2 n}} e^{2n\theta} \abs{\sum_{j=1,2}I_n(0, an, B_j)}^2 \mathrm da =  \frac{1}{8}\int_{y_1}^{y_2} (1+a^2) f_K(a; \theta) \mathrm d a.
\end{align}
\subsubsection{Rigorous justification}\label{sec: rigorous_justificaiton}

\paragraph{Justifying the replacement of \texorpdfstring{$\ceil{an}$}{[an]} with \texorpdfstring{$an$}{an}.}
 Because of the assumption that $a\in (-\cos \theta +\epsilon, \cos\theta -\epsilon)$, we can in fact have the large $n$ asymptotic relation \eqref{eqns: four_term_decomposition} with $a$ replaced by $a_n:=\ceil{an}/n$, which is uniform across all choices of $a$. The precise argument has been made in Proposition 2.2 of \cite{SunadaTate_2012}.  Nonetheless, the formal procedure of using $a$ and $a_n$ yield similar computations.

\paragraph{Justifying the limits in \texorpdfstring{\eqref{eqns: case_0_integral_asymptotic_computation}}{}.}
The convergence 
\begin{align}
     \lim_{n\to \infty}\int_{\frac{1}{n}\floor{y_1 n}}^{\frac{1}{n}\floor{y_2 n}}f_K(a; \theta)(1+a)^2\cos(n\pi(1-a)) \mathrm d a=0,
\end{align}
is a direct application of the Riemann-Lebesgue Lemma by noting that $f_K(a; \theta)(1+a)^2$ is integrable over $[y_1, y_2]$. 
Next, for the convergence:
\begin{align}\label{eqns: 2nd_term_convergence_RL}
  && \lim_{n\to \infty}\int_{\frac{1}{n}\floor{y_1 n}}^{\frac{1}{n}\floor{y_2 n}}f_K(a; \theta)(1+a)^2\sgn(a) \sin (2n\Theta_{B_1}(k_1, a)-n\pi(1-a)) \mathrm d a=0, 
\end{align}
 without loss of generality, we consider $0\leq y_1 < y_2$ or $y_1< y_2\leq 0$ (otherwise we split into two integrals). Then the Riemann-Lebesgue Lemma follows because of the integrability of $f_K(a; \theta)(1+a)^2\sgn(a)$ and the fact that
\begin{align}
 &\inf_{a\in (y_1, y_2)}  \abs{\frac{d}{da}[2\Theta_{B_1}(k_1, a) - \pi(1-a)]} \\ \notag 
   &=\inf_{a\in (y_1, y_2)}\abs{2\Big[\frac{\cos \theta \sin k_1}{\sqrt{1-\cos^2\theta \cos^2k_1 }}k_1'(a) - ak_1'(a) -k_1(a)\Big] + \pi}\\
&= 2\inf_{a\in (y_1, y_2)}\abs{\frac{\pi}{2}-k_1(a)}>0,  \notag 
\end{align}
 since $y_1, y_2\in (-\cos \theta+\epsilon , \cos \theta - \epsilon)$. For the last convergence in \eqref{eqns: case_0_integral_asymptotic_computation}, by the sum-to-product formula, we have
 \begin{align}
&\lim_{n\to \infty}\int_{\frac{1}{n}\floor{y_1 n}}^{\frac{1}{n}\floor{y_2 n}}f_K(a; \theta)(1+a)^2\sgn(a) \sin (2n\Theta_{B_1}(k_1, a)-n\pi(1-a))\cos(n\pi(1-a))\mathrm d a\\ \label{eqns: RL_more_difficult_term1}
&\qquad =\frac{1}{2}\lim_{n\to \infty}\int_{\frac{1}{n}\floor{y_1 n}}^{\frac{1}{n}\floor{y_2 n}}f_K(a; \theta)(1+a)^2\sgn(a) \sin (2n\Theta_{B_1}(k_1, a))\mathrm d a\\ \label{eqns: RL_more_difficult_term2}
&\qquad +\frac{1}{2}\lim_{n\to \infty} \int_{\frac{1}{n}\floor{y_1 n}}^{\frac{1}{n}\floor{y_2 n}}f_K(a; \theta)(1+a)^2\sgn(a) \sin (2n\Theta_{B_1}(k_1, a)-2n\pi(1-a))\mathrm d a. 
\end{align}
In fact both limits \eqref{eqns: RL_more_difficult_term1} and \eqref{eqns: RL_more_difficult_term2} vanish due to similar reasons. We consider only the limit \eqref{eqns: RL_more_difficult_term1}; assume without loss of generality that $y_1\leq 0\leq y_2$ (the case $y_1, y_2$ with the same signs can be obtained by adding two integrals of the previous type). Next, we split the limit \eqref{eqns: RL_more_difficult_term1} into
\begin{align}\label{eqns: RL_stationary_phase_term}
    &\frac{1}{2}\lim_{n\to \infty}\int_{0}^{\frac{1}{n}\floor{y_2 n}}f_K(a; \theta)(1+a)^2\sgn(a) \sin (2n\Theta_{B_1}(k_1, a)) \mathrm d a\\ \label{eqns: RL_non_stationary_phase_term}
    & \, \qquad {} +  \frac{1}{2}\lim_{n\to \infty} \int_{\frac{1}{n}\floor{y_1 n}}^{0}f_K(a; \theta)(1+a)^2\sgn(a) \sin (2n\Theta_{B_1}(k_1, a))\mathrm d a. 
\end{align}
For $\eqref{eqns: RL_stationary_phase_term}$, we note that
\begin{align}
    \frac{d}{da}\Theta_{B_1}(k_1, a) = k_1(a) = 0\quad \mbox{for}\quad a\in [0, y_2) \quad \mbox{implies that}\quad a = 0.
\end{align}
Therefore, $a=0$ is the unique critical point to $\Theta_{B_1}(k_1(a), a)$ over the domain of integration. By the method of stationary phase \eqref{eqns: stationary_phase_heuristic}, the integral \eqref{eqns: RL_stationary_phase_term} is of order $O(n^{-1/2})$ so the limit vanishes. For $\eqref{eqns: RL_non_stationary_phase_term}$, the limit vanishes similarly to \eqref{eqns: 2nd_term_convergence_RL} by noting that 
\begin{align}
    \inf_{a\in {(y_2, 0)}}\abs{\frac{d}{da}\Theta_{B_1}(k_1, a)} =  \inf_{a\in {(y_2, 0)}} \abs{k_1(a)}>0,
\end{align}
since $k_1(a)$ is in the third quadrant as seen in its definition in \ref{eqns: critical_point_definition}. 

\subsection{The remaining two limits}
The rigorous justifications of the limits  \eqref{eqns: integral_quantum_limit_1} and \eqref{eqns: integral_quantum_limit_cross} are similar to the case of \eqref{eqns: integral_quantum_limit_0}. In this section we will be only showing the additional formal computations that are needed to verify the two limits. 

\subsubsection{Formal computation for \texorpdfstring{\eqref{eqns: integral_quantum_limit_1}}{}}
The only difference between the asymptotic relation for $e^{n\theta}\sum_{j=1, 2}I_n(1, an, B_j)$ and that of $e^{n\theta}\sum_{j=1, 2}I_n(0, an, B_j)$ as seen in \eqref{eqns: four_term_decomposition} are the eigenvector terms, which are given by:
\begin{equation}\label{eqns: case1_eigenvectors}
\begin{split}
        r_{B_1}(k, 0) \ell_{B_1}(k, 0) &= \frac{1}{4i}\frac{\cos k \sin \theta+i\sin k \sin \theta}{\sqrt{1-\cos^2 k \cos^2 \theta}}, \\
      r_{B_2}(k, 0) \ell_{B_2}(k, 0) &= -\frac{1}{4i}\frac{\cos k \sin \theta+i\sin k \sin \theta}{\sqrt{1-\cos^2 k \cos^2 \theta}}. 
\end{split}
\end{equation}

Their values at the critical points could be simplified by noting that:

\begin{equation}\label{eqns: case_1_eigenvectors_critical_points}
\begin{split}
    \frac{\cos k_1\sin \theta}{\sqrt{1-\cos^2 k_1 \cos^2 \theta}} &= \sgn(a)\sec\theta \sqrt{\cos^2 \theta - a^2},\\  \frac{\sin k_1\sin \theta}{\sqrt{1-\cos^2 k_1 \cos^2 \theta}} &= a\tan\theta. 
\end{split}
\end{equation}
Using \eqref{eqns: 2nd_derivative_0} and \eqref{eqns: case1_eigenvectors}  with the corresponding values at the saddle points given by \eqref{eqns: 2nd_derivative_at_saddle_0} and \eqref{eqns: case_1_eigenvectors_critical_points} shows that $e^{n\theta}\sum_{j=1, 2}I_n(1, an, B_j)$ is asymptotically equivalent to the sum: 
\begin{align}\label{eqns: case_1_1st_part}
     & \qquad \frac{1}{2\pi}\frac{\sgn(a) \sqrt{\cos^2\theta - a^2}\sec \theta}{4i}\sqrt{\frac{2\pi^2}{n} f_K(a; \theta)}\\ \notag
\times & \{e^{in \cos^{-1}(\cos k_1 \cos \theta)-ik_1an + i\sgn(a)\frac{\pi}{4}} - e^{in \cos^{-1}(\cos k_2 \cos \theta)-ik_2an - i\sgn(a)\frac{\pi}{4}}\\ 
&-e^{-in \cos^{-1}(\cos k_3 \cos \theta)-ik_3an + i\sgn(a)\frac{\pi}{4}} + e^{-in \cos^{-1}(\cos k_4 \cos \theta)-ik_4an - i\sgn(a)\frac{\pi}{4}}\} \notag\\ \notag
+&\\ \label{eqns: case_1_2nd_part}
& \qquad \frac{1}{2\pi}\frac{a\tan \theta}{4}\sqrt{\frac{2\pi^2}{n} f_K(a; \theta)}\\\notag
\times & \{e^{in \cos^{-1}(\cos k_1 \cos \theta)-ik_1an + i\sgn(a)\frac{\pi}{4}} + e^{in \cos^{-1}(\cos k_2 \cos \theta)-ik_2an - i\sgn(a)\frac{\pi}{4}}\\
&+e^{-in \cos^{-1}(\cos k_3 \cos \theta)-ik_3an + i\sgn(a)\frac{\pi}{4}} + e^{-in \cos^{-1}(\cos k_4 \cos \theta)-ik_4an - i\sgn(a)\frac{\pi}{4}}\}.  \notag
\end{align}
The oscillation term in \eqref{eqns: case_1_1st_part} is given by 
\begin{equation*}
\begin{split}
  & 2 \cos (n\Theta_{B_1}(k_1, a) + \sgn(a)\frac{\pi}{4}) -  2\cos (n\Theta_{B_1}(k_2, a) -\sgn(a)\frac{\pi}{4})\\ \notag 
   =& -4 \sin (\frac{n\pi (1-a)}{2})\sin(n \Theta_{B_1}(k_1, a)-\frac{n\pi (1-a)}{2} +\sgn(a)\frac{\pi}{4})\\ \notag 
 =&{} -\frac{4}{\sqrt{2}}\sin(\frac{n\pi (1-a)}{2})\Big[\sin(n\Theta_{B_1}(k_1, a) - \frac{n\pi(1-a)}{2}) \\
 & {}+\sgn(a)\cos(n\Theta_{B_1}(k_1, a) - \frac{n\pi(1-a)}{2}) \Big], 
\end{split}
\end{equation*}
while the oscillation term of \eqref{eqns: case_1_2nd_part} is given by 
\begin{equation*}
\begin{split}
  & 2 \cos (n\Theta_{B_1}(k_1, a) + \sgn(a)\frac{\pi}{4}) +  2\cos (n\Theta_{B_1}(k_2, a)-\sgn(a)\frac{\pi}{4})\\  
   =&4 \cos (\frac{n\pi (1-a)}{2})\cos(n \Theta_{B_1}(k_1, a)-\frac{n\pi (1-a)}{2} +\sgn(a)\frac{\pi}{4})\\  
 =&\frac{4}{\sqrt{2}}\cos(\frac{n\pi (1-a)}{2})\Big[\cos(n\Theta_{B_1}(k_1, a) - \frac{n\pi(1-a)}{2}) \\ & \qquad \qquad \qquad \qquad \qquad \qquad {} - \sgn(a)\sin(n\Theta_{B_1}(k_1, a)  - \frac{n\pi(1-a)}{2})\Big].  
\end{split}
\end{equation*}
Therefore, $e^{n\theta}\sum_{j=1, 2} I_n(1, an, B_j))$ is in turn asymptotically equivalent to: 
\begin{equation}
\begin{split}\label{eqns: full_single_stationary_phase_case1}
\quad &i\frac{\sgn(a)\sqrt{\cos^2 \theta - a^2}\sec\theta}{2\sqrt{n}} \sqrt{f_K(a; \theta)} \sin(\frac{n\pi (1-a)}{2}) \\ 
    \quad &\quad \times \Big[\sin(n\Theta_{B_1}(k_1, a) - \frac{n\pi(1-a)}{2}) +\sgn(a)\cos(n\Theta_{B_1}(k_1, a) - \frac{n\pi(1-a)}{2})\Big]\\ 
    \quad &+\frac{a \tan \theta}{2\sqrt{n}}\sqrt{f_K(a; \theta)} \cos(\frac{n\pi (1-a)}{2}) \\
    \quad &\quad \times \Big[\cos(n\Theta_{B_1}(k_1, a) - \frac{n\pi(1-a)}{2}) - \sgn(a)\sin(n\Theta_{B_1}(k_1, a) - \frac{n\pi(1-a)}{2})\Big]. 
\end{split}
\end{equation}
Taking modulus square on \eqref{eqns: full_single_stationary_phase_case1} gives the asymptotic equivalence of $e^{2n\theta}\lvert{\sum_{j=1, 2}I_n(1, an, B_j)}\rvert$: 
\begin{align}\label{eqns: full_single_stationary_phase_case1_square}
    &\frac{(\cos^2 \theta - a^2)\sec^2\theta}{4n } f_K(a; \theta)\\ \notag 
    &\times \frac{1}{2}(1-\cos(n\pi(1-a)))[1+\sgn (a)\sin (2n\Theta_{B_1}(k_1, a)-n\pi(1-a))]
    \\ \notag 
    &+\frac{a^2 \tan^2 \theta}{4n}f_K(a; \theta)\\
    &\times \frac{1}{2}(1+\cos(n\pi(1-a)))[1-\sgn (a)\sin (2n\Theta_{B_1}(k_1, a)-n\pi(1-a))]. \notag 
\end{align}
Note that integrals preserve asymptotic equivalence just like the previous case \eqref{eqns: case_0_integral_asymptotic_computation}.  Using the Riemann-Lebesgue Lemma, we have similar to the formal computations in \eqref{eqns: case_0_integral_asymptotic_computation} that:
\begin{align*}
   & \lim_{n\to \infty}n \int _{\frac{1}{n}\floor{n y_1}}^{\frac{1}{n}\floor{n y_2}}e^{2n\theta}\abs{\sum_{j=1, 2}I_n(1, an, B_j)} \mathrm d a \\
   &= \int_{y_1}^{y_2}\frac{1}{8}f_K(a; \theta)((\cos^2\theta - a^2)\sec^2\theta + a^2 \tan^2\theta) \mathrm d a 
   =\frac{1}{8}\int_{y_1}^{y_2}(1-a^2)f_K(a; \theta) \mathrm d a.\notag 
\end{align*}
\subsubsection{Formal computation for \texorpdfstring{\eqref{eqns: integral_quantum_limit_cross}}{(3.36)}}
Using the asymptotic equivalence  obtained in the two equations \eqref{eqns: full_single_stationary_phase} and \eqref{eqns: full_single_stationary_phase_case1}, we can now deduce the asymptotic equivalence of the product $e^{2n\theta}(\sum_{j=1, 2}I_n(0, an, B_j))\overline{(\sum_{j=1, 2}I_n(1, an, B_j))}$ to be:
\begin{align*}
    &\frac{(1+a)a\tan\theta }{4n }f_K(a; \theta) [1-\sgn(a) \sin (2n\Theta_{B_1}(k_1, a))] \cos^2 \Big(\frac{n\pi(1-a)}{2}\Big)\\
    -&i\frac{(1+a)\sqrt{\cos^2\theta - a^2}\sec \theta}{4n}f_K(a; \theta) \cos(2n\Theta_{B_1}(k_1, a)-n\pi (1-a)). \notag 
\end{align*}
Again, by the preservation of asymptotic equivalence with integrals and using the Riemann-Lebesgue Lemma, we have similar to \eqref{eqns: case_0_integral_asymptotic_computation} that:
\begin{align*}
   &\lim_{n\to \infty}  n\int_{\frac{1}{n}\floor{ny_1}}^{\frac{1}{n}\floor{ny_2}}  e^{2n\theta}\left(\sum_{j=1, 2} I_n(0, an, B_j))(\overline{\sum_{j=1, 2} I_n(1, an, B_j)}\right)\mathrm d a \\
     &=  \frac{1}{8}\int_{y_1}^{y_2}(1+a)a\tan\theta f_K(a; \theta)\mathrm d a. \notag 
\end{align*}